\documentclass[11pt,a4paper]{amsart}

\usepackage{amsmath}
\usepackage{amsfonts}
\usepackage{graphicx}
\usepackage{framed}
\usepackage{amssymb,cancel}
\usepackage{xcolor}
\usepackage{esint}
\usepackage{comment}

\usepackage{etex}
\usepackage{latexsym}
\usepackage[english]{babel}
\usepackage[utf8x]{inputenc}

\def \R {\mathbb{R}}

\def \de {\partial}

\usepackage{amsthm}
\theoremstyle{definition}
\newtheorem{definition}{Definition}[section]

\newtheorem{remark}[definition]{Remark}

\theoremstyle{plain}
\newtheorem{theorem}[definition]{Theorem}
\newtheorem{proposition}[definition]{Proposition}
\newtheorem{lemma}[definition]{Lemma}
\newtheorem{corollary}[definition]{Corollary}
\newcommand{\Cemb}{C_{\rm emb}}
\newcommand{\Sr}{S_r}
\newcommand{\Es}{\mathcal E_s}
\newcommand{\norm}[1]{\left\lVert #1\right\rVert}

\numberwithin{equation}{section}
\usepackage[colorlinks=true,urlcolor=blue,
citecolor=red,linkcolor=blue,linktocpage,pdfpagelabels,
bookmarksnumbered,bookmarksopen]{hyperref}
\usepackage[hyperpageref]{backref}

\makeatletter
\@namedef{subjclassname@2020}{\textup{2020} Mathematics Subject Classification}
\makeatother

\begin{document}

\title[Local-nonlocal Lane-Emden equation]{Lane-Emden equation for the superposition of a local and nonlocal operator with the ``wrong sign''}

  \author[S.\,Biagi]{Stefano Biagi}
 %\author[S.\,Dipierro]{Serena Dipierro}
 %\author[E.\,Valdinoci]{Enrico Valdinoci}
 \author[E.\,Vecchi]{Eugenio Vecchi}
 
 \address[S.\,Biagi]{Dipartimento di Matematica
 	\newline\indent Politecnico di Milano \newline\indent
 	Via Bonardi 9, 20133 Milano, Italy}
 \email{stefano.biagi@polimi.it}
 
% \address[S.\,Dipierro]{Department of Mathematics and Statistics
% 	\newline\indent University of Western Australia \newline\indent
% 	35 Stirling Highway, WA 6009 Crawley, Australia}
% \email{serena.dipierro@uwa.edu.au}
% 
% \address[E.\,Valdinoci]{Department of Mathematics and Statistics
% 	\newline\indent University of Western Australia \newline\indent
% 	35 Stirling Highway, WA 6009 Crawley, Australia}
% \email{enrico.valdinoci@uwa.edu.au}
 
 \address[E.\,Vecchi]{Dipartimento di Matematica
 	\newline\indent Università di Bologna \newline\indent
 	Piazza di Porta San Donato 5, 40126 Bologna, Italy}
 \email{eugenio.vecchi2@unibo.it}



%\thanks{The Authors are
%	members of the {\em Gruppo Nazionale per
%	l'Analisi Ma\-te\-ma\-ti\-ca, la Probabilit\`a e le loro Applicazioni}
%	(GNAMPA) of the {\em Istituto Nazionale di Alta Matematica} (INdAM). S.B. and E.V. are partially supported by the
%	Indam-GNAMPA project {\em Esistenza, regolarità e studio delle proprietà qualitative per problemi nonlineari}.\\
%	The Authors thank B. Ruffini for the fruitful discussions on the topics of \cite{MMR}.}
	
\date{\today}

\begin{abstract} 
	We consider the Lane-Emden equation associated with the operator $	-\Delta  - \gamma \, (-\Delta)^s$ (for positive $\gamma$). First we consider the equation in the entire space and we prove that 
       no Serrin-type critical exponent arises in the $H^1(\mathbb R^n)$ framework.
	 Secondly we consider the problem settled in an open, bounded and connected set and we prove the existence of a positive solution for sufficiently small $s$.
\end{abstract}

\maketitle

\section{Introduction} 
We consider either the following problem
\begin{align}\label{eq:Problem_Spazio}
	\left\{\begin{array}{rl}
		-\Delta u - \gamma \, (-\Delta)^s u = |u|^{r-2}u & \textrm{ in } \mathbb{R}^n,\\
		u\not\equiv0 & \textrm{ in } \mathbb{R}^n,		
	\end{array}\right.
\end{align}
\noindent or its version settled in $\Omega \subset \mathbb{R}^n$ (with $n \geq 3$) being an open and bounded set with Lipschitz boundary $\partial \Omega$,
\begin{align}\label{eq:Problem}
	\left\{\begin{array}{rl}
		-\Delta u - \gamma \, (-\Delta)^s u = |u|^{r-2}u & \textrm{ in } \Omega,\\
		u\not\equiv0 & \textrm{ in } \Omega,\\
		u= 0 & \textrm{ in } \mathbb{R}^{n} \setminus \Omega.
	\end{array}\right.
\end{align}
Throughout the paper, we strengthen the regularity assumptions on $\partial\Omega$
whenever needed. The operator $(-\Delta)^s$ is the fractional Laplacian,
defined, for $s\in(0,1)$ and $u$ smooth enough, by
\[
 (-\Delta)^s u(x):=
 \int_{\mathbb R^n}
 \frac{2u(x)-u(x+y)-u(x-y)}{|y|^{n+2s}}\,dy.
\]
In the whole-space problem, $\gamma>0$ is arbitrary and no upper restriction
on the power exponent will be imposed. In the bounded-domain variational
problem, i\-n\-stead, we assume
\[
 0<\gamma<C_{\rm emb}(s)\qquad\text{and}\qquad 1<r\le 2^*:=\frac{2n}{n-2},
\]
where $C_{\rm emb}(s)$ is the optimal constant in the inequality
\[
 C_{\rm emb}(s)[u]_s^2\le \|\nabla u\|_{L^2(\Omega)}^2,
 \qquad u\in H_0^1(\Omega),
\]
see, precisely, \eqref{eq:Cemb}. Concerning the constant $C_{\rm emb}(s)$, we anticipate here the terminology
that will be used throughout the paper. As shown in Section~2,
$C_{\rm emb}(s)$ is the smallest value of $\lambda$ for which the generalized
eigenvalue problem
\[
 \begin{cases}
 -\Delta\Phi=\lambda(-\Delta)^s\Phi & \text{in }\Omega,\\
 \Phi=0 & \text{in }\mathbb R^n\setminus\Omega
 \end{cases}
\]
admits a nonzero weak solution
(in the variational sense specified
in Section \ref{sec:Prel}).
Such values of $\lambda$ are called \emph{generalized eigenvalues}, and
the corresponding nonzero weak solutions are called
\emph{generalized eigenfunctions}. For each fixed generalized eigenvalue,
the vector space of all weak solutions, including zero, is its
\emph{associated eigenspace}. 
In particular, we say that $C_{\rm emb}(s)$ is
\emph{simple} if its associated eigenspace is one-dimensional.

This terminology always refers to the above generalized problem;
Dirichlet eigenvalues and eigenfunctions of $-\Delta$ will instead refer
to the classical Dirichlet Laplacian. 
\vspace{0.1cm}

As it is clear, both \eqref{eq:Problem} and \eqref{eq:Problem_Spazio} can be interpreted as Lane-Emden-type equations for the operator
\begin{equation}
	\mathcal L_{\gamma,s}:=-\Delta-\gamma(-\Delta)^s, \qquad \gamma >0.
\end{equation}	

Before entering into the details of the results we will prove, let us mention that  $\mathcal L_{\gamma,s}$ fits in the more general framework of linear superposition of (possibly) fractional operators given by
\begin{equation*}
	L_{\mu} := \int_{[0,1]}(-\Delta)^s \, d\mu(s),
\end{equation*}
\noindent where $\mu = \mu^+ - \mu^-$ is a signed measure. This class of operators has been recently studied in several directions, see e.g. \cite{RE2, RE22, RE5, RE6, RE7, RE8, PESP:2024, RE12}. We also mention \cite{MMV, GMV, CCMV} where the authors refer to $\mathcal L_{\gamma,s}$ as {\em indefinite} or {\em nonpositive} operator.
This recent line of research dealing with the sum of two operators in which the lower-order one possesses the wrong sign (w.r.t. the validity of a maximum principle) is somehow complementary to the more classical study of operators of mixed integer and fractional order (see e.g.~\cite{COURR,Cancelier} in relation with the maximum principle) which is also experiencing an intense revival (see e.g., among the many,\cite{CKSV,BFV,RSPoho,BDVV-CPDE22,AC,AntoCozzi,DeFMin,GarainKinnunen,GarainLindgren,BEMV}). All this interest can be easily understood in view of the possible applications to biology and other real--life situations, see e.g.~\cite{DV-niche,DPLV2}.

\medskip

\noindent -\,\,\emph{Motivations}. 
The classical Lane-Emden equation reads
\begin{equation}\label{eq:LANE-EMDEN}
	-\Delta u = u^p,
\end{equation}
\noindent and it has been the subject of several lines of research. One of the interesting phenomena related 
to this PDE is the existence of a critical exponent in the following sense: the numbers
\begin{equation*}
	\inf \left\{p\in \mathbb{R}: \text{$\exists$ a positive solution to \eqref{eq:LANE-EMDEN} in $\R^n$}\right\},
\end{equation*}
\noindent and, considering a bounded, smooth and star-shaped set $\Omega$,
\begin{equation*}
	\sup \left\{p\in \mathbb{R}: \text{$\exists$ a positive solution $u$ to \eqref{eq:LANE-EMDEN} in $\Omega$ with $u = 0$ on $\de\Omega$}\right\},
\end{equation*}
\noindent do coincide with $2^{\ast}-1$, see \cite{GidasSpruck,CaGiSp}.
We briefly mention that a totally analogous phenomenon appears in the fractional counterpart of \eqref{eq:LANE-EMDEN}, namely replacing $-\Delta$ with $(-\Delta)^s$, the only difference being that the precise value of the critical exponent is $2^{\ast}_s -1 = \tfrac{n+2s}{n-2s}$.

A natural question related to the study of the operator of mixed order given by 
$$\mathcal{L}:= -\Delta + (-\Delta)^s, \qquad s \in (0,1),$$
\noindent has been whether a similar notion of {\it critical exponent} does exist for $\mathcal{L}$. A first answer to this question has been provided in \cite{BaDPQ}, where the authors showed that even if it exists, it cannot coincide with the one of $-\Delta$. One can probably trace the origin of this wrong belief to the paper \cite{BDVV5}. The starting point therein is the following simple observation: let $\Omega \subseteq \mathbb{R}^{n}$ be an open set and $s \in (0,1)$, then there exists a constant $\mathcal{S}_{n,s}(\Omega)>0$ such that
  \begin{equation}\label{eq:mixed_Sobolev}
	  	\mathcal{S}_{n,s}(\Omega)  \|u\|^{2}_{L^{2^{\ast}}(\mathbb{R}^{n})} \leq \|\nabla u\|^{2}_{L^{2}(\mathbb{R}^{n})} + [u]^2_{s},
	  \end{equation}
  for every 
  $u \in C^{\infty}_{0}(\Omega)$. This is clearly a consequence of the classical Sobolev inequality.
  
In principle, the constant $\mathcal{S}_{n,s}(\Omega)$ could depend on~$n$, $s$ and~$\Omega$, but this is actually not the case (see~\cite[Theorem~1.1]{BDVV5}), indeed 
  \begin{equation*}
	  	\mathcal{S}_{n,s}(\Omega) = S_{n}
	  \end{equation*}
   where $S_n$ is the best constant for the classical Sobolev inequality, achieved by the Aubin-Talenti functions.

  Moreover, the energy functional 
  \begin{equation}\label{eq:Mixed_Energy}
	  	E(u) := \dfrac{1}{2} \int_{\R^n}|\nabla u|^2 \, dx + \dfrac{1}{2}[u]^{2}_s
	  \end{equation}
exhibits a lack of scaling invariance, namely
  \begin{equation}\label{eq:Lack_Scaling}
	  	E(u_{t}) = \dfrac{1}{2} \int_{\R^n}|\nabla u|^2 \, dx + \dfrac{t^{2s-2}}{2}[u]^{2}_s, 
	  \end{equation}
where $u_t(x):=t^{(n-2)/2}u(tx)$, with $u$ viewed as a function on
$\mathbb R^n$, is the rescaling which preserves the Dirichlet seminorm
$\|\nabla u\|_{L^2(\mathbb R^n)}$. 
This second simple remark allows to show that the best constant in the local-nonlocal Sobolev inequality \eqref{eq:mixed_Sobolev} is never achieved, see~\cite[Theorem~1.2]{BDVV5}.

Since the first variation of the energy functional $E$ in \eqref{eq:Mixed_Energy} is given by the local-nonlocal operator $\mathcal{L}$,
 it became natural to attack the study of the existence of positive solutions to Sobolev critical problems in bounded domains (in the spirit of the celebrated \cite{BrNir}) and with leading operator $\mathcal{L}$: a not exhaustive list of contributions may include~\cite{BDVV5,BiagiVecchi,BiagiVecchi2}. 

\medskip 

Let us now come back to the topic of interest of the present paper.
The first result is the following 
\begin{theorem}\label{thm:intro-whole-space}
	Let $n\ge3$, $s\in(0,1)$ and $\gamma>0$. Any nonnegative weak supersolution
	$u\in H^1(\R^n)$ of
	\begin{equation*}
		-\Delta u-\gamma(-\Delta)^s u\ge0
		\qquad\text{in }\R^n
	\end{equation*}
	is identically zero. In particular, for every $r>1$, the equation
	\begin{equation*}
		-\Delta u-\gamma(-\Delta)^s u=u^{r-1}
		\qquad\text{in }\R^n
	\end{equation*}
	admits no nontrivial nonnegative weak solution in $H^1(\R^n)$, whenever the
	right-hand side is locally integrable.
\end{theorem}

In this case the wrong sign of the fractional term yields a rather rigid Liouville
property. The argument is based on a positive test function on which the
operator $\mathcal L_{\gamma,s}$ has a strictly negative sign at a sufficiently large scale. In particular, no
restriction related to the Sobolev critical exponent is involved. Our technique here is clearly reminiscent of the one by Mitidieri and Pohozaev (see e.g. \cite{MiPo} and the references therein for an accurate description of the method and the several applications found by the authors) and it has also been recently applied in a local--nonlocal parabolic setting in \cite{BiPuVe} to detect the corresponding Fujita exponent.

As a final comment, let us also compare our Theorem \ref{thm:intro-whole-space} with the recent results appeared in \cite{GuoLiXie} (for the linear operator $-\Delta + (-\Delta)^s$) and \cite{DPPQ} (for the sum of possibly fractional $p$-Laplacians): in both cases, the authors proved that the Serrin exponent coincides with the one related to the operator of the lowest order, e.g. $\tfrac{n}{n-2s}$ in the case of \cite{GuoLiXie}.

\medskip

The second part of the paper is devoted to the study of \eqref{eq:Problem} settled in a bounded set $\Omega$. 
We postpone to Section~\ref{sec:Prel}, and in particular to \eqref{eq:def_X12}, the exhaustive discussion on the function space $\mathcal{X}^{1,2}(\Omega)$ and several equivalent definitions. For the moment, we just point out that
one of the equivalent definitions reads
$$\mathcal{X}^{1,2}(\Omega) = \big\{u\in H^1(\R^n):\,\text{$u|_\Omega\in H_0^1(\Omega)$ and 
	$u\equiv 0$ a.e.\,in $\R^n\setminus\Omega$}\big\}.$$
A first result in the Sobolev critical case (i.e. for $r=2^{\ast}$) has been recently obtained in \cite{BDVV6} and we believe it is worth briefly recalling it.
Let us first define the functional
\begin{equation}\label{eq:Def_Qgamma}
	Q_{\gamma}(u) := \int_{\Omega}|\nabla u|^{2} \, dx - \gamma \, [u]^{2}_{s}, \quad u \in \mathcal{X}^{1,2}(\Omega),
\end{equation}
and let us consider the following constrained minimization problem
\begin{equation}\label{eq:Def_Sgamma}
	S(\gamma) := \inf \big\{ Q_{\gamma}(u) : u \in \mathcal{X}^{1,2}(\Omega), \|u\|_{L^{2^{\ast}}(\Omega)}=1 \big\}.
\end{equation}

The main result obtained in \cite{BDVV6}
reads as follows:

\begin{theorem}[\cite{BDVV6} - Theorem 1.1]\label{thm:Da_BLMS}
	Let $n\geq 3$ and $\Omega \subset \mathbb{R}^{n}$ be an open and bounded set with Lipschitz boundary.
	%Let further be $\gamma \in (0,C_{emb})$.
	Then, the following assertions hold:
	\begin{enumerate}
			\item[(1)] If $n\geq 5$, the infimum in \eqref{eq:Def_Sgamma} is achieved for every $\gamma\in (0,\Cemb(s))$. 
			\vspace{0.1cm}
			
			\item[(2)] If $n = 3,4$, there exists $\gamma^*\in [0,\Cemb(s))$ such that
			the infimum in \eqref{eq:Def_Sgamma} is achieved for every~$\gamma\in(\gamma^*,\Cemb(s))$.
		\end{enumerate}
	In both cases, if $u_0\in \mathcal{X}^{1,2}(\Omega)$ achieves the infimum in \eqref{eq:Def_Sgamma}, then
	\begin{equation} \label{eq:tildeuSolutiondaMinimo}
			\widetilde{u} = S(\gamma)^{1/(2^*-2)}u_0		
		\end{equation}
	is a nontrivial weak solution to problem \eqref{eq:Problem} with $r=2^{\ast}$.
\end{theorem}

As it is clear, the sign of the solution has not been investigated, and it can neither be a simple consequence of a maximum principle nor be based on a trivial test function argument (i.e. testing the equation by $u^+$). To address this issue, we follow a path recently used in \cite{RE8} to prove,
for small fractional orders, the sign-definiteness of generalized
eigenfunctions associated with the least eigenvalue; see Section \ref{sec:Prel} for more details. We stress that this approach is kind of transversal and naturally appears in other problems, see e.g. \cite{MMR}, where the authors study a shape optimization problem associated with a functional which exhibits competitive terms like in our case. 

%The strategy adopted can be schematized as follows:
%\begin{itemize}
%	\item 
%\end{itemize}
As the argument in \cite{RE8} suggests, there is actually no reason to restrict our interest only to the Sobolev critical case $r=2^{\ast}$. Therefore, 
for $0<\gamma<\Cemb(s)$ and $1<r\le2^*$, we define
\begin{equation}\label{eq:Sr}
	\Sr(\gamma):=\inf\left\{Q_\gamma(u):u\in \mathcal{X}^{1,2}(\Omega),\
	\norm{u}_{L^r(\Omega)}=1\right\}.
\end{equation}
It is pretty easy to see that the infimum in \eqref{eq:Sr} is achieved in the Sobolev subcritical case $r \in (1,2^{\ast})$ because no lack of compactness occurs (see Proposition \ref{prop:attainment} for the standard details). 

Once this is established, we have the following
\begin{theorem}\label{thm:intro-threshold}
	Let $n\ge3$, let $\Omega\subset\R^n$ be open and bounded with Lipschitz
	boundary, let $s\in(0,1)$ and let $1<r\le2^*$. For $\gamma$ in the range
	where the infimum defining $\Sr(\gamma)$ is achieved, let $u_{\gamma,r}$
	be any normalized minimizer. If $r<2^*$ this range is the whole interval
	$(0,\Cemb(s))$, while for $r=2^*$ the attainment range is the one given by
	\cite[Theorem~1.1]{BDVV6}. Then the following assertions hold.
	\begin{enumerate}
		\item[i)]
		\begin{equation*}
			\Sr(\gamma)\longrightarrow0
			\qquad\text{as }\gamma\to\Cemb(s)^-.
		\end{equation*}
		\item[ii)] If $\gamma_k\to\Cemb(s)^-$ and $u_{\gamma_k,r}$ are normalized
		minimizers, then, up to a subsequence, there exists
		$u_\infty\in \mathcal{X}^{1,2}(\Omega)$ such that
		\begin{equation*}
			u_{\gamma_k,r}\longrightarrow u_\infty
			\qquad\text{strongly in }\mathcal{X}^{1,2}(\Omega),
		\end{equation*}
		and
		\begin{equation*}
			\norm{u_\infty}_{L^r(\Omega)}=1,
			\qquad
			\norm{\nabla u_\infty}_{L^2(\Omega)}^2
			=\Cemb(s)[u_\infty]_s^2.
		\end{equation*}
		In particular, $u_\infty$ is an $L^r$-normalized generalized eigenfunction
associated with $C_{\rm emb}(s)$. If $C_{\rm emb}(s)$ is simple and $\Phi_s$
is any fixed $L^r$-normalized generalized eigenfunction associated with it,
then
		\begin{equation}\label{eq:intro-plus-minus}
			u_\infty=\Phi_s\qquad\text{or}\qquad u_\infty=-\Phi_s.
		\end{equation}
		\item[iii)] For every $R\in[1,+\infty)$ there exist
		$\gamma_R<\Cemb(s)$ and $M_R>0$ such that
		\begin{equation*}
			\norm{u_{\gamma,r}}_{L^R(\Omega)}\le M_R
		\end{equation*}
		for every $\gamma\in(\gamma_R,\Cemb(s))$ and every normalized minimizer
		$u_{\gamma,r}$.
	\end{enumerate}
\end{theorem}

We notice that the first part i) of the above Theorem \ref{thm:intro-threshold} has been already noticed for $r=2^{\ast}$ in \cite{BDVV6}. Regarding assertion~ii), the normalized minimizers converge, along
subsequences, to generalized eigenfunctions associated with
$C_{\rm emb}(s)$. If this eigenvalue is simple, the limit is determined
up to sign, as in \eqref{eq:intro-plus-minus}. Under the additional connectedness and boundary
regularity assumptions of Lemma~\ref{lem:DPLSV-small-s}, the corresponding $L^r$-normalized
generalized eigenfunction $\Phi_s$ can be chosen positive for sufficiently
small $s$. To exploit this property, we need to strengthen the convergence.
This is the content of the next result.

 \medskip
 
\begin{theorem}\label{thm:intro-C1}
	Assume in addition that $0<s<1/2$ and that $\Omega$ is of class $C^{1,1}$.
	Then, for every fixed $q\in[2,+\infty)$, there exist
	$\gamma_q<\Cemb(s)$ and $C_q>0$ such that
	\begin{equation*}
		\norm{u_{\gamma,r}}_{W^{2,q}(\Omega)}\le C_q
	\end{equation*}
	for every $\gamma\in(\gamma_q,\Cemb(s))$ and every normalized minimizer
	$u_{\gamma,r}$. In particular, if $q>n$ and $\alpha=1-n/q$, then
	\begin{equation*}
		\norm{u_{\gamma,r}}_{C^{1,\alpha}(\overline\Omega)}\le C_q
	\end{equation*}
	for every $\gamma\in(\gamma_q,\Cemb(s))$ and every normalized minimizer
	$u_{\gamma,r}$.
	
	Moreover, if $\gamma_k\to\Cemb(s)^-$ and $u_{\gamma_k,r}\to u_\infty$
	in $\mathcal{X}^{1,2}(\Omega)$ as in Theorem~\ref{thm:intro-threshold}, then
	\begin{equation*}
		u_{\gamma_k,r}\longrightarrow u_\infty
		\qquad\text{strongly in }W^{2,q}(\Omega)
	\end{equation*}
	for every fixed $q\in[2,+\infty)$. In particular, choosing $q>n$,
	\begin{equation*}
		u_{\gamma_k,r}\longrightarrow u_\infty
		\qquad\text{strongly in }C^1(\overline\Omega).
	\end{equation*}
	If $C_{\rm emb}(s)$ is simple, then $u_\infty=\pm\Phi_s$ as in~\eqref{eq:intro-plus-minus}.
\end{theorem}
\begin{remark}
The simplicity assumption appearing in the last part of Theorems~\ref{thm:intro-threshold}
and~\ref{thm:intro-C1} is satisfied under suitable additional hypotheses. Indeed,
if $\Omega$ is connected and of class $C^{1,1}$, Lemma~
\ref{lem:DPLSV-small-s} below shows
that there exists $s_0\in(0,1/2)$ such that, for every $s\in(0,s_0)$,
$C_{\rm emb}(s)$ is simple and admits a strictly positive generalized
eigenfunction. Hence, in this small-$s$ regime, the conclusions in
Theorems~\ref{thm:intro-threshold}
and~\ref{thm:intro-C1} involving $\pm\Phi_s$ hold automatically.
\end{remark}

We stress that the $C^{1,1}$-regularity assumed in Theorem
\ref{thm:intro-C1} is used
to apply the global boundary regularity theory from \cite{RE8}. Once the
strong $C^1$ convergence is available, we can combine it with the small-$s$
spectral information above and the Hopf lemma to obtain the desired sign.
\begin{theorem}\label{thm:intro-positivity}
	Assume that $\Omega$ is connected and of class $C^{1,1}$, and let
	$1<r\le2^*$. There exists $s_1\in(0,1/2)$ such that, for every fixed
	$s\in(0,s_1)$, there exists
	\begin{equation*}
		\overline\gamma_{r,s}<\Cemb(s)
	\end{equation*}
	with the following property: if
	$\gamma\in(\overline\gamma_{r,s},\Cemb(s))$, then every normalized
	minimizer $u_{\gamma,r}$ has a constant sign in $\Omega$. After multiplication
	by $-1$ if necessary,
	\begin{equation*}
		u_{\gamma,r}>0\qquad\text{in }\Omega.
	\end{equation*}
\end{theorem}
The connectedness assumption is needed in order to apply the small-$s$ spectral result of Lemma \ref{lem:DPLSV-small-s}. Indeed, on disconnected domains the spectral picture for the generalized eigenvalue problem introduced above may be substantially different. 
In particular, the result in \cite[Theorem~1.4]{RE8} shows that generalized
eigenfunctions associated with $\Cemb(s)$ may be forced to
change sign on disconnected domains, thus preventing the argument used
here to transfer positivity from $\Phi_s$ to the normalized minimizers.

Once Theorem \ref{thm:intro-positivity} is established, 
from the attainment result in Proposition \ref{prop:attainment} we immediately deduce the following

\begin{corollary}\label{cor:intro-critical}
	Assume the hypotheses of Theorem~\ref{thm:intro-positivity},
	and let
$1<r\le2^*$, $r\neq2$. For every $s>0$ sufficiently small, and for every
	$\gamma<\Cemb(s)$ sufficiently close to $\Cemb(s)$, the problem
	\begin{equation}\label{eq:intro-critical-problem}
		\begin{cases}
			-\Delta v - \gamma(-\Delta)^s v = |v|^{r-2}v&\text{in }\Omega,\\
			v=0&\text{in }\R^n\setminus\Omega
		\end{cases}
	\end{equation}
	admits a positive weak solution. More precisely, if $u_{\gamma,r}$ is a
	positive normalized minimizer, then
	\begin{equation*}
		v_\gamma:=S_{r}(\gamma)^{1/(r-2)}u_{\gamma,r}
	\end{equation*}
	is a positive weak solution of~\eqref{eq:intro-critical-problem}.
\end{corollary}
 
\medskip

The above Corollary \ref{cor:intro-critical} holds for every
$r\in(1,2^*],\,r\neq 2$
and shows that, at least for sufficiently small
$s$, the existence of positive solutions to the Lane--Em\-den equation driven
by $\mathcal L_{\gamma,s}$ is markedly different from the case of the operator
$$-\Delta+(-\Delta)^s.$$
However, if one wonders about the existence of a critical exponent for $\mathcal{L}_{\gamma,s}$ in bounded domains, this question remains unanswered. 
In this regard, a first instance of a different  behaviour also once supercritical Sobolev exponents are considered, comes from the following 

\begin{theorem}\label{thm:intro-supercritical}
	Let $\Omega \subset \mathbb{R}^n$ be an open, bounded, connected and star-shaped set with $C^2$-smooth boundary $\partial \Omega$. Let us denote by $\nu$ the outer unit normal to $\partial \Omega$. 
	Let $s \in (0,1/2)$ and let $p > \tfrac{n+2}{n-2}$.
	
	Then, there are no nontrivial solutions
$u\in C^2(\overline\Omega)\cap C(\mathbb R^n)$ to
	\begin{equation}\label{eq:Non_Esistenza_Limitati}
		\begin{cases}
			-\Delta u - \gamma\, (-\Delta)^s u =|u|^{p-1}u &\text{in }\Omega,\\
			u=0&\text{in }\R^n\setminus\Omega,
		\end{cases}
	\end{equation}
	\noindent provided that
	\begin{equation*}
		0 < \gamma < \Cemb \, \left(\dfrac{\tfrac{n-2}{2}-\tfrac{n}{p+1}}{\tfrac{n-2s}{2}-\tfrac{n}{p+1}} \right).	\end{equation*}
\end{theorem}

\medskip

Notice that the upper bound 
\begin{equation*}
	\Cemb \, \left(\dfrac{\tfrac{n-2}{2}-\tfrac{n}{p+1}}{\tfrac{n-2s}{2}-\tfrac{n}{p+1}} \right),
\end{equation*}
\noindent can be close to $\Cemb$ if $s$ is close to 1, but this is out of the range for which we proved existence of positive solutions for Sobolev critical and subcritical exponents, or in big dimensions $n$.
We refer to Section~\ref{sec:Proof_1.5} for further comments on Theorem~\ref{thm:intro-supercritical}, especially
regarding the assumptions on $s$ and the regularity required in the
Pohozaev identity.

\medskip 
 
\noindent-\,\,\emph{Plan of the paper}. 
The paper is organized as follows. In Section~\ref{sec:Prel} we collect the functional
setting and the preliminary results, including the spectral information from
\cite{RE8} that is crucial for the positivity argument. 

In Section~\ref{sec:whole-space} we prove the
whole-space Liouville property stated in Theorem~\ref{thm:intro-whole-space} by constructing a
suitable rescaling of the fractional Sobolev bubble. 

Section~\ref{sec:Proof_1.3} is devoted to
the variational problem in bounded domains: we prove attainment in the
subcritical case, convergence of the normalized minimizers to the eigenspace associated
with $C_{\rm emb}(s)$ as $\gamma\to C_{\rm emb}(s)^-$, and uniform finite Lebesgue
bounds, thereby proving Theorem~\ref{thm:intro-threshold}. 

In Section~\ref{sec:Proof_1.4} we establish the global
$W^{2,q}$ estimates and the strong $C^1$ convergence of Theorem~\ref{thm:intro-C1}.

Finally, Section~\ref{sec:Proof_1.5} contains the positivity argument leading to Theorem~\ref{thm:intro-positivity}, and concludes with the Pohozaev-based supercritical
nonexistence result of Theorem~\ref{thm:intro-supercritical}.

\bigskip

\noindent-\,\,\textbf{Acknowledgements.} The Authors are
members of the {\em Gruppo Nazionale per
	l'Analisi Ma\-te\-ma\-ti\-ca, la Probabilit\`a e le loro Applicazioni}
(GNAMPA) of the {\em Istituto Nazionale di Alta Matematica} (INdAM). S.B. and E.V. are partially supported by the
Indam-GNAMPA project {\em Esistenza, regolarità e studio delle proprietà qualitative per problemi nonlineari}.\\
The Authors thank B. Ruffini for the fruitful discussions on the topics of \cite{MMR}.

\noindent-\,\,\textbf{AI disclosure.} The authors used GPT-6 to assist with the detection of typographical errors, consistency checks, and the verification of bibliographic references. All mathematical arguments and proofs were developed, written, and verified by the authors, who take full responsibility for the entire content of the manuscript.

\noindent-\,\,\textbf{Conflict of interest.} The authors state no conflict of interest.

\noindent-\,\,\textbf{Data availability statement.} There are no data associated with this research.
  
  \section{Preliminary results}\label{sec:Prel}
  \noindent {\bf Notations.} Throughout the paper, we tacitly
  exploit all the notation listed below; we thus refer the Reader to this list
  for any non-standard notation encountered.
  \begin{itemize}
  	\item $\Omega\subset\mathbb R^n$ is an open and bounded set with Lipschitz
boundary $\partial\Omega$. Connectedness, as well as stronger boundary
regularity, will be assumed explicitly when needed.
\vspace{0.05cm}

  	\item Given any set $E$, $|E|$ denotes its $n$-dimensional Lebesgue measure.
  	\vspace{0.05cm}
  	
  	\item Given any set $E$, $E^c$ denotes its complement (in $\mathbb{R}^n$), namely~$E^c:=\R^n\setminus E$.
  	\vspace{0.05cm}
  	
  	\item Given $x_0 \in \mathbb{R}^n$ and $r>0$, we denote by $B_{r}(x_0)$ the Euclidean open ball of center $x_0$ and radius $r$. When $x_0=0$, we will omit it and simply write~$B_r$.
  	\vspace{0.05cm}
  	
  	\item Given any two vectors $v,w \in \mathbb{R}^{n}$, we denote by
  	$v\cdot w$ the usual Euclidean scalar product.
  	\vspace{0.05cm}
  	
  	\item Given any $p \in [1,n)$ we denote by $p^{*}$ the associated Sobolev exponent (with respect to Euclidean space $\mathbb{R}^n$), that is
  	$$p^{*}  := \frac{np}{n-p}.$$
  	\item Given any $r \in (1,+\infty)$, we denote by $r'$ the conjugate exponent of $r$ in the usual H\"older inequality, that is, 
  	$$r'  := \frac{r}{r-1}.$$
  	\item Given a reflexive Banach space $V$ with dual space $V'$, we denote by 
  	$\langle \cdot,\cdot\rangle$
  	the duality product. 
  	\vspace{0.05cm}
  	
  	\item Given a measurable set $E$, we denote by $\chi_{E}$ its characteristic function.
  	\vspace{0.05cm}
  	
  	\item Given any function $v$, we denote the positive and negative part of $v$ by 
  	\begin{equation}\label{eq:def_positive_negative_part}
  		v^{+}:= \max\{v,0\} \quad \textrm{ and } \quad v^{-}:= -\min\{v,0\}.
  	\end{equation}
%  	\item $S_n$ will denote the best constant in the classical Sobolev inequality which is achieved by the family of Aubin-Talenti functions defined, for every~$y\in\R^n$ and~$\varepsilon>0$, as \begin{equation}\label{eq:def_Talentiane}
%  		V_{\varepsilon,y}(x) := \dfrac{\varepsilon^{(n-2)/2}}{(\varepsilon^{2} + |x-y|^2)^{(n-2)/2}}.
%  	\end{equation}
%  	We will avoid writing the dependence on $y$ when $y=0$.
  \end{itemize}
  
  \medskip 
  
  \noindent \,\,\textbf{Functional setting.} We now set the adequate functional setting to study problem \eqref{eq:Problem}. We refer to e.g.~\cite[Section~2]{BDVV5} for more details.
  
  Let~$s\in (0,1)$ and let $u:\R^n\to\R$ be a measurable function. We set
  \begin{equation} \label{eq:seminorm}
  	[u]_s := \left(\iint_{\R^{2n}}\frac{|u(x)-u(y)|^2}{|x-y|^{n+2s}}\,dx\,dy\right)^{1/2},
  \end{equation}
  and we refer to $[u]_s$ as the \emph{Gagliardo seminorm} of $u$ (of order $s$).
  Let us also introduce the bilinear form
  \begin{equation}\label{eq:Es}
  	\Es(u,v):=\iint_{\R^{2n}}
  	\frac{(u(x)-u(y))(v(x)-v(y))}{|x-y|^{n+2s}}\,dx\,dy,
  \end{equation}
  so that $[u]_s^2= \Es(u,u)$. 		
  
  Let $\varnothing\neq \Omega\subseteq\R^n$ (here $n\geq 3$ would suffice) be an open set, not necessarily bounded. We define the function space
  $\mathcal{X}^{1,2}(\Omega)$ as 
  \begin{equation*}
  	\mathcal{X}^{1,2}(\Omega) := \overline{C_0^\infty(\Omega)}^{\,\,\rho(\cdot)}
  \end{equation*}
   i.e. as  the com\-ple\-tion
  of $C_0^\infty(\Omega)$ with respect to the norm  
  \begin{equation}\label{eq:def_rho}
  	\rho(u) := \left(\|\nabla u\|^2_{L^2(\R^n)}+[u]^2_s\right)^{1/2}.
  \end{equation}
    We stress that the norm $\rho(\cdot)$ is induced by the scalar product
  		$$\langle u,v\rangle_{\rho} := \int_{\R^n}\nabla u\cdot\nabla v\,dx
  		+ \iint_{\R^{2n}}\frac{(u(x)-u(y))(v(x)-v(y))}{|x-y|^{n+2s}}\,dx\,dy,$$
  		and therefore $\mathcal{X}^{1,2}(\Omega)$ is a Hilbert space.
  		
  It is known, see e.g.~\cite{DPV}, that the Sobolev space $H^{1}(\mathbb{R}^n)$ continuously embeds into the fractional Sobolev space~$H^{s}(\mathbb{R}^n)$ for~$s\in (0,1)$, i.e., there exists a positive constant $C>0$ (depending on $s$ and $n$) such that
 \begin{equation}\label{eq:embedding}
 	C \, [u]^{2}_{s}\leq  \|u\|^2_{H^{1}(\mathbb{R}^n)}, 
 \end{equation}
 for every $u \in H^{1}(\mathbb{R}^n)$.

For our purposes, we will also be interested in the case~$\Omega$ bounded. Recalling the continuous embedding given by~\eqref{eq:embedding}, and combining it with the classical Poincar\'e inequality, we find that
  $\rho(\cdot)$ and the full $H^1$-norm in $\R^n$
  are actually equivalent on the space $C^\infty_0(\Omega)$, and hence
  \begin{equation}\label{eq:def_X12}
  	\begin{aligned}
  	\mathcal{X}^{1,2}(\Omega) & = \overline{C_0^\infty(\Omega)}^{\,\,\|\cdot\|_{H^1(\R^n)}} \\
  	& = \big\{u\in H^1(\R^n):\,\text{$u|_\Omega\in H_0^1(\Omega)$ and 
  		$u\equiv 0$ a.e.\,in $\R^n\setminus\Omega$}\big\}.
  \end{aligned}
\end{equation}

 We recall the following standard facts.

\begin{remark}\label{rem:embeddings}
	The space $\mathcal{X}^{1,2}(\Omega)$ is continuously embedded into $L^p(\Omega)$ for
	every $1\le p\le2^*$, and compactly embedded whenever
	$1\le p<2^*$, where
	$$2^* = 2n/(n-2)$$
	Moreover, by~\cite[Theorem~1.3]{emBE} (applied with the two
	Sobolev exponents equal to $2$ and orders $1$ and $s$),
	we also have the compact embedding
	\begin{equation*}
		\mathcal{X}^{1,2}(\Omega)\hookrightarrow H^s(\R^n).
	\end{equation*}
	We denote by $S_n$ the best constant in the classical Sobolev inequality,
	so that
	\begin{equation}\label{eq:Sobolev}
		S_n\norm{u}_{L^{2^*}(\Omega)}^2
		\le \norm{\nabla u}_{L^2(\Omega)}^2.
	\end{equation}
\end{remark}

\medskip

We now give the precise definition of weak solution to \eqref{eq:Problem}.

\begin{definition}\label{def:WeakSol}
	We say that a function $u:\mathbb{R}^n\to\mathbb{R}$ is a \emph{weak solution}
	to problem~\eqref{eq:Problem} if it satisfies the following properties:
	\begin{itemize}
		\item[1)] $u\in\mathcal{X}^{1,2}(\Omega)$ and $u\not\equiv 0$ in $\Omega$;
		\item[2)] for every test function $v\in \mathcal{X}^{1,2}(\Omega)$, we have
		\begin{align*}
			\int_{\Omega}\nabla u\cdot\nabla v\,dx
			-\gamma \, \iint_{\R^{2n}}\frac{(u(x)-u(y))(v(x)-v(y))}{|x-y|^{n+2s}}\,dx\,dy = \int_\Omega |u|^{r-2}uv\,dx.	
		\end{align*}
	\end{itemize}
\end{definition} 

\begin{remark}
For $0<\gamma<C_{\rm emb}(s)$ (see \eqref{eq:Cemb} below), we stress that in~\cite[equation~(2.10)]{RE7} the
authors use a different norm, modeled on the energy associated with the
local--nonlocal operator. In our notation, their choice reads
\[
 \bigl(\|\nabla u\|_{L^2(\mathbb R^n)}^2-\gamma[u]_s^2\bigr)^{1/2}.
\]
It turns out that their choice is actually equivalent to the one taken in the present paper, thanks to~\cite[Lemma~2.1]{RE2}.
\end{remark}

 %This also follows directly from the compact embedding
 %  $\mathcal{X}^{1,2}(\Omega)\hookrightarrow H^s(\R^n)$ recalled below.
 
% The assumption $\gamma \in (0, \Cemb)$ can now be explained: we are interested in dealing with a local and nonlocal operator {\it with the wrong sign}, namely
% \begin{equation}\label{eq:Operator_wrong_sign}
% 	-\Delta  - \gamma \, (-\Delta)^s,
% \end{equation}
% but at the same time we need some restrictions on the coefficient $\gamma$ so that the operator is positive definite.
 
 \medskip 
 
 \noindent \textbf{A generalized eigenvalue problem.} 
 Starting from~\eqref{eq:embedding} and applying the
classical Poincar\'e inequality, we can easily infer the existence of a positive constant ~$C > 0$ (depending on $\Omega$ as well) such that
 \begin{equation} \label{eq:Cemb-ineq} 
 	C[{u}]^2_s\leq \int_\Omega|\nabla u|^2\,dx = \|\nabla u\|^2_{L^2(\Omega)},	
 \end{equation} for every $u\in H_0^1(\Omega)$. Hence, we can define
 \begin{equation}\label{eq:Cemb}
 	\Cemb(s) := \inf \big\{ \|\nabla u\|^{2}_{L^2(\Omega)} : u\in H_0^1(\Omega),\,[u]_{s}^2 =1\big\}.
 \end{equation}
 As observed in~\cite[Remark~2.1(iii)]{BDVV6}, the infimum in~\eqref{eq:Cemb}
 is achieved. When no confusion arises, we will simply write $\Cemb$ therefore avoiding to repeat the dependence on $s$.
 The Euler equation associated with~\eqref{eq:Cemb} reads
 \begin{equation}\label{eq:first-eigenproblem}
 	\int_\Omega \nabla\Phi\cdot\nabla\varphi\,dx
 	=\Cemb \, \Es(\Phi,\varphi)
 	\qquad\text{for every }\varphi\in \mathcal{X}^{1,2}(\Omega),
 \end{equation}
 for every minimizer $\Phi$ of~\eqref{eq:Cemb}. 
 As usual, \eqref{eq:first-eigenproblem} can be seen as the weak formulation (in $\mathcal{X}^{1,2}(\Omega)$) of \[
 -\Delta\Phi=\Cemb(-\Delta)^s\Phi\quad\text{in }\Omega,
 \qquad
 \Phi=0\quad\text{in }\R^n\setminus\Omega.
 \]
 The latter is a particular instance of a problem recently studied in \cite{RE8}
and referred to therein as a nonlocal eigenvalue problem. To justify the
terminology anticipated in the Introduction, let $\lambda$ be a generalized
eigenvalue and let 
$$\Phi\in \mathcal{X}^{1,2}(\Omega)\setminus\{0\}$$ 
be an associated
generalized eigenfunction. 

Testing its weak equation with $\Phi$ gives
\[
 \lambda
 =\frac{\|\nabla\Phi\|_{L^2(\Omega)}^2}{[\Phi]_s^2}
 \ge C_{\rm emb}.
\]
Since the infimum in \eqref{eq:Cemb} is attained and its minimizers satisfy \eqref{eq:first-eigenproblem},
$C_{\rm emb}$ is the least generalized eigenvalue. Its associated
eigenspace consists precisely of the functions $\Phi\in \mathcal{X}^{1,2}(\Omega)$
satisfying
\[
 \|\nabla\Phi\|_{L^2(\Omega)}^2=C_{\rm emb}[\Phi]_s^2.
\]
In particular, the minimizers of \eqref{eq:Cemb} are precisely the elements of this
eigenspace satisfying $[\Phi]_s=1$.  
 \vspace{0.1cm}
 
 The next lemma collects some results obtained in \cite{RE8}. 
 To be more precise, it is a direct consequence of the perturbative spectral results
 in~\cite[Theorems~1.2 and~1.3 and the proof of Theorem~1.3]{RE8}, together
 with the global regularity estimate in~\cite[Theorem~1.8]{RE8}. 
 
 \begin{lemma}\label{lem:DPLSV-small-s}
 	Let $\Omega$ be open, bounded, connected and of class $C^{1,1}$. There exists
$s_0\in(0,1/2)$ such that, for every $s\in(0,s_0)$, $\Cemb$ is a simple generalized eigenvalue of the problem
\begin{equation} \label{eq:small-s-eigenproblem}
 -\Delta\Phi=\lambda(-\Delta)^s\Phi\quad\text{in }\Omega,
 \qquad
 \Phi=0\quad\text{in }\mathbb R^n\setminus\Omega,
\end{equation}
and every generalized eigenfunction associated with $\Cemb$
is either strictly positive or strictly negative in $\Omega$.

Moreover, for every fixed $r\in(1,2^*]$, let $\varphi_1$ be the positive
first Dirichlet eigenfunction of $-\Delta$, normalized by
$\|\varphi_1\|_{L^r(\Omega)}=1$, and let $\Phi_s$ be the positive
generalized eigenfunction associated with $C_{\rm emb}$, normalized
by $\|\Phi_s\|_{L^r(\Omega)}=1$. Then
\begin{equation}\label{eq:Phi-to-phi1}
 \Phi_s\to \varphi_1
 \qquad\text{strongly in }C^1(\overline\Omega)
 \quad\text{as }s\to0^+.
 	\end{equation}
 \end{lemma}
 
\begin{remark}\label{rmk:Sulla_Prova_di_RE8}
Let us try to give a truly rough idea of what lies behind the result proved in~\cite{RE8}.
The key observation is that, with the normalization adopted in~\cite{RE8}, the generalized eigenvalue problem 
\eqref{eq:small-s-eigenproblem} can be viewed as a perturbation of the classical Dirichlet
eigenvalue problem as $s\to0^+$: the normalized generalized eigenfunctions corresponding, in our convention,
to $C_{\rm emb}$ converge, after a suitable choice of sign, to the positive
    first Dirichlet eigenfunction of $-\Delta$. 
    
    The remaining
part of the argument consists in improving this convergence up to $C^1$ and then
transferring the positivity of the first Dirichlet eigenfunction
     to the corresponding generalized eigenfunctions for small $s$,
     after a suitable choice of sign. 
     
     The relation between the normalization used in~\cite{RE8} and ours is
discussed in the next remark.
\end{remark}

 \begin{remark}\label{rem:small-s-normalization}
 The normalization of the fractional Laplacian used in 
 \cite{RE8} differs from
ours by the positive factor $c_{n,s}$. Writing
\[
 (-\Delta)^s_{\rm ref}:=c_{n,s}(-\Delta)^s,
\]
we have, for every fixed $s\in(0,1)$,
\[
 -\Delta\Phi=\lambda(-\Delta)^s\Phi
 \quad\Longleftrightarrow\quad
 -\Delta\Phi=\frac{\lambda}{c_{n,s}}(-\Delta)^s_{\rm ref}\Phi,
\]
with the same exterior condition and in the same weak sense.
Thus the generalized eigenvalues are rescaled by $c_{n,s}^{-1}$, whereas
the associated eigenspaces are unchanged. In particular, the least
generalized eigenvalue in the convention of
\cite{RE8} is
$C_{\rm emb}/c_{n,s}$. Its simplicity and the sign properties of its
associated generalized eigenfunctions therefore carry over to our
convention, as does the convergence in \eqref{eq:Phi-to-phi1} under the same
$L^r$-normalization and choice of sign. No convergence statement for the
eigenvalues themselves is being transferred without this rescaling.
  \end{remark}

\medskip

\noindent \,\,\textbf{A regularity estimate.}  
 We end this section proving the following elementary lower-order estimate, valid on the whole space. Despite this, it will be useful both for \eqref{eq:Problem} and \eqref{eq:Problem_Spazio}.
 
 \begin{lemma}\label{lem:lower-order}
 	There exists a constant $C=C(n,s)>0$ such that, for every $v\in H^1(\R^n)$
 	and every $\rho>0$,
 	\begin{equation}\label{eq:lower-order}
 		[v]_s^2\le
 		C\rho^{2-2s}\norm{\nabla v}_2^2
 		+C\rho^{-2s}\norm v_2^2.
 	\end{equation}
 	In particular, for every $\varepsilon>0$ there exists $C_\varepsilon>0$
 	such that
 	\begin{equation}\label{eq:lower-order-epsilon}
 		[v]_s^2\le\varepsilon\norm{\nabla v}_2^2
 		+C_\varepsilon\norm v_2^2
 		\qquad\text{for every }v\in H^1(\R^n).
 	\end{equation}
 \end{lemma}
 
 \begin{proof}
 	Writing $h=y-x$ in the definition of the fractional seminorm, we have
 	\[
 	[v]_s^2=
 	\int_{\R^n}\frac{\norm{v(\cdot+h)-v}_2^2}{|h|^{n+2s}}\,dh.
 	\]
 	For $|h|\le\rho$, the fundamental theorem of calculus gives
 	\[
 	\norm{v(\cdot+h)-v}_2\le |h|\norm{\nabla v}_2,
 	\]
 	whereas, for every $h\in\R^n$,
 	\[
 	\norm{v(\cdot+h)-v}_2\le2\norm v_2.
 	\]
 	Splitting the integral into $|h|\le\rho$ and $|h|>\rho$ yields
 	\eqref{eq:lower-order}. Formula~\eqref{eq:lower-order-epsilon} follows by
 	choosing $\rho$ sufficiently small.
 \end{proof}
 
 \begin{remark}\label{rem:lower-order-X}
 	We explicitly observe that, since every $v\in \mathcal{X}^{1,2}(\Omega)$ belongs to $H^1(\R^n)$ and vanishes
 	a.e. in $\R^n\setminus\Omega$, Lemma~\ref{lem:lower-order} applies in
 	particular to every $v\in \mathcal{X}^{1,2}(\Omega)$. 
 	
 	Hence, for every
 	$\varepsilon>0$, there exists $C_\varepsilon>0$ such that
 	\begin{equation*}
 		[v]_s^2\le\varepsilon\norm{\nabla v}_{L^2(\Omega)}^2
 		+C_\varepsilon\norm v_{L^2(\Omega)}^2
 		\qquad\text{for every }v\in \mathcal{X}^{1,2}(\Omega).
 	\end{equation*}
 	Thus the lower-order interpolation estimate used below on the bounded domain is
 	simply the restriction to $\mathcal{X}^{1,2}(\Omega)$ of the whole-space estimate
 	\eqref{eq:lower-order-epsilon}.
 \end{remark}

 \section{Proof of Theorem \ref{thm:intro-whole-space}}\label{sec:whole-space}
In this section we prove Theorem \ref{thm:intro-whole-space}. We first fix the notion of weak supersolu\-tion
used below and then construct a positive function on which $\mathcal L_{\gamma,s}$
becomes strictly negative after a suitable large-scale rescaling. 

Throughout the
section, $n\ge3$, $s\in(0,1)$ and $\gamma>0$ are fixed.
 \begin{definition}
 	We say
 	that $u\in H^1(\R^n)$ is a 
 	\emph{nonnegative weak supersolution} of the equation
 	\begin{equation}\label{eq:whole-space-supersolution}
 		\mathcal L_{\gamma,s}u = 0
 		\qquad\text{in }\R^n
 	\end{equation}
 	if $u\ge0$ a.e. in $\R^n$ and
 	\begin{equation*}
 		\int_{\R^n}\nabla u\cdot\nabla\varphi\,dx
 		-\gamma \, \Es(u,\varphi)\ge0\qquad\forall\,\,\varphi\in C_0^\infty(\R^n),\,\varphi\geq 0. 
 	\end{equation*}
 \end{definition}
 We now set
  \begin{equation}\label{eq:whole-space-beta}
 	\beta:=\frac{n-2s}{2}
 \end{equation}
 and define
 \begin{equation}\label{eq:whole-space-psi}
 	\psi(x):=\frac{1}{(1+|x|^2)^\beta}
 	=\frac{1}{(1+|x|^2)^{(n-2s)/2}}.
 \end{equation}
The latter is a scalar multiple of what is usually called {\it fractional Sobolev bubble}. Therefore, with the choice of $\beta$ made in~\eqref{eq:whole-space-beta}, we have both a known sign and decay of $(-\Delta)^{s}\psi$.

 \begin{lemma}\label{lem:whole-space-barrier}
 	There exists a constant $a_{n,s}>0$, depending only on the normalization of the
 	fractional Laplacian, such that
 	\begin{equation}\label{eq:fractional-bubble}
 		(-\Delta)^s\psi(x)
 		=a_{n,s}(1+|x|^2)^{-(n+2s)/2}
 		\qquad\text{for every }x\in\R^n.
 	\end{equation}
 	In particular, $(-\Delta)^s\psi>0$ in $\R^n$. Moreover, there exists
 	$C_0=C_0(n,s)>0$ such that
 	\begin{equation}\label{eq:local-vs-fractional-psi}
 		-\Delta\psi\le C_0(-\Delta)^s\psi
 		\qquad\text{in }\R^n.
 	\end{equation}
 \end{lemma}
 
 \begin{proof}
 	The explicit formula~\eqref{eq:fractional-bubble}
 	follows from
 	\cite[formulas~(6.5)--(6.6)]{FerrariVerbitsky}. We briefly spell out the specialization,
 	also in order to keep track of the sign. In the notation of~\cite{FerrariVerbitsky},
 	consider
 	\[
 	f_{\widetilde\beta}(x)
 	:=-(1+|x|^2)^{-\widetilde\beta/2},
 	\qquad
 	\widetilde\beta:=n-2s,
 	\]
 	and take $\alpha:=2s$. Formula~(6.5) in~\cite{FerrariVerbitsky} gives
 	\[
 	-(-\Delta)^{\alpha/2}f_{\widetilde\beta}(x)
 	=
 	C(\alpha,\widetilde\beta,n)
 	{}_2F_1\left(a,b;c;-|x|^2\right),
 	\]
 	where $C(\alpha,\widetilde\beta,n)>0$ and
 	\[
 	a=\frac{n+\alpha}{2}=\frac{n+2s}{2},
 	\qquad
 	b=\frac{\alpha+\widetilde\beta}{2}=\frac n2,
 	\qquad
 	c=\frac n2.
 	\]
 	Since $b=c$, the elementary identity
 	\[
 	{}_2F_1(a,c;c;z)=(1-z)^{-a}
 	\]
 	yields
 	\[
 	{}_2F_1\left(
 	\frac{n+2s}{2},\frac n2;\frac n2;-|x|^2
 	\right)
 	=
 	(1+|x|^2)^{-(n+2s)/2}.
 	\]
 	Recalling that $f_{\widetilde\beta}=-\psi$, we obtain
 	\[
 	(-\Delta)^s\psi(x)
 	=
 	a_{n,s}(1+|x|^2)^{-(n+2s)/2},
 	\]
 	for a suitable constant $a_{n,s}>0$. The value of $a_{n,s}$ depends on
 	the normalization chosen for the fractional Laplacian; only its positivity
 	will be relevant below.
 	
 	We now compute the local part. More generally, if
 	$\psi_\beta(x):=(1+|x|^2)^{-\beta}$, then
 	\begin{equation*}
 		\Delta\psi_\beta(x)
 		=\frac{2\beta}{(1+|x|^2)^{\beta+2}}
 		\left[(2\beta-n+2)|x|^2-n\right].
 	\end{equation*}
 	For $\beta=(n-2s)/2$ this yields
 	\begin{equation}\label{eq:laplacian-special-psi}
 		-\Delta\psi(x)
 		=(n-2s)
 		\frac{n-2(1-s)|x|^2}
 		{(1+|x|^2)^{(n+4-2s)/2}}.
 	\end{equation}
 	Combining~\eqref{eq:fractional-bubble} and
 	\eqref{eq:laplacian-special-psi}, we can rewrite the latter identity as
 	\begin{equation*}
 		-\Delta\psi(x)
 		=\frac{n-2s}{a_{n,s}}
 		\frac{n-2(1-s)|x|^2}{(1+|x|^2)^{2(1-s)}}
 		(-\Delta)^s\psi(x).
 	\end{equation*}
 	Since
 	\[
 	\frac{n-2(1-s)|x|^2}{(1+|x|^2)^{2(1-s)}}\le n
 	\qquad\text{for every }x\in\R^n,
 	\]
 	we obtain~\eqref{eq:local-vs-fractional-psi}, for instance with
 	\begin{equation*}
 		C_0:=\frac{n(n-2s)}{a_{n,s}}.
 	\end{equation*}
 	This ends the proof.
 \end{proof}
 
 We now consider a suitable rescaling of the function $\psi$. To this aim, for $R>0$ we set
 \begin{equation*}
 	\psi_R(x):=\psi\left(\frac{x}{R}\right).
 \end{equation*}
 As already mentioned, $-\Delta$ and $(-\Delta)^s$ scale differently. More precisely,
 \begin{equation}\label{eq:whole-space-scaling}
 	-\Delta\psi_R(x)
 	=R^{-2}(-\Delta\psi)\left(\frac{x}{R}\right),
 	\qquad
 	(-\Delta)^s\psi_R(x)
 	=R^{-2s}(-\Delta)^s\psi\left(\frac{x}{R}\right).
 \end{equation}
 Consequently, Lemma~\ref{lem:whole-space-barrier} gives
 \begin{align}
 	\mathcal L_{\gamma,s}\psi_R(x)
 	&=R^{-2}(-\Delta\psi)\left(\frac{x}{R}\right)
 	-\gamma R^{-2s}(-\Delta)^s\psi\left(\frac{x}{R}\right)\notag\\
 	&\le R^{-2s}\left(C_0R^{2s-2}-\gamma\right)
 	(-\Delta)^s\psi\left(\frac{x}{R}\right).
 	\label{eq:barrier-scaled-estimate}
 \end{align}
 Since $s<1$, we have $R^{2s-2}\to0$ as $R\to+\infty$. Hence, there exists
 $R_0=R_0(n,s,\gamma)>0$ such that
 \begin{equation}\label{eq:R0-choice}
 	C_0R^{2s-2}\le\frac\gamma2
 	\qquad\text{for every }R\ge R_0.
 \end{equation}
 It follows from~\eqref{eq:fractional-bubble} and
 \eqref{eq:barrier-scaled-estimate} that, for every $R\ge R_0$,
 \begin{equation}\label{eq:barrier-negative}
 	\mathcal L_{\gamma,s}\psi_R(x)
 	\le-\frac\gamma2 R^{-2s}
 	(-\Delta)^s\psi\left(\frac{x}{R}\right)<0
 	\qquad\text{for every }x\in\R^n.
 \end{equation}
 The function $\psi_R$ cannot always be directly used as a test function. Indeed,
 taking into account~\eqref{eq:whole-space-psi} we have
 \begin{equation}\label{eq:psi-decay}
 	\psi_R(x)\asymp |x|^{-(n-2s)}
 	\qquad\text{as }|x|\to+\infty,
 \end{equation}
 and therefore
 \begin{equation*}
 	\psi_R\in L^2(\R^n)
 	\quad\Longleftrightarrow\quad n>4s.
 \end{equation*}
 No restriction of this type is needed, however. The relevant observation is that
 $\psi_R$ has finite homogeneous local and nonlocal energies, and can be
 approximated by compactly supported nonnegative functions in both energies.
 
 \begin{lemma}\label{lem:whole-space-cutoff}
 	For every $R>0$, we have
 	\begin{equation*}
 		\|\nabla \psi_R\|_{L^2(\R^n)},\,[\psi_R]_s<+\infty.
 	\end{equation*}
 	Moreover, there exists a sequence
 	$\varphi_k\in C_0^\infty(\R^n)$, $\varphi_k\ge0$, such that
 	\begin{equation}\label{eq:cutoff-convergence}
 		\norm{\nabla(\varphi_k-\psi_R)}_2\to 0,
 		\qquad
 		[\varphi_k-\psi_R]_s\to 0
 		\qquad\text{as }k\to+\infty.
 	\end{equation}
 \end{lemma}
 
 \begin{proof}
 	The fact that $[\psi_R]_s<+\infty$ is the standard finite-energy
 	property of the fractional Sobolev bubble in~\eqref{eq:whole-space-psi}.
 	As for the local energy,
 	\eqref{eq:whole-space-psi} gives, for $|x|$ large,
 	\begin{equation*}
 		|\nabla\psi_R(x)|\le C_R|x|^{-(n-2s+1)}.
 	\end{equation*}
 	Thus $\nabla\psi_R\in L^2(\R^n)$, since
 	\[
 	2(n-2s+1)>n
 	\quad\Longleftrightarrow\quad n+2>4s,
 	\]
 	and the latter inequality follows from $n\ge3$ and $s<1$.
 	\vspace{0.1cm}
 	
 	Let now $\eta\in C_0^\infty(\R^n)$ be such that
 	\begin{equation*}
 		0\le\eta\le1,
 		\qquad
 		\eta\equiv1\ \text{in }B_1,
 		\qquad
 		\eta\equiv0\ \text{in }\R^n\setminus B_2,
 	\end{equation*}
 	and define
 	\begin{equation*}
 		\eta_k(x):=\eta\left(\frac{x}{k}\right),
 		\qquad
 		\varphi_k(x):=\eta_k(x)\psi_R(x).
 	\end{equation*}
 	Clearly, $\varphi_k\in C_0^\infty(\R^n)$ and $\varphi_k\ge0$.
 	
 	We first consider the local energy. Since
 	\[
 	\nabla(\varphi_k-\psi_R)
 	=(\eta_k-1)\nabla\psi_R+\psi_R\nabla\eta_k,
 	\]
 	the first term converges to zero in $L^2(\R^n)$ by dominated convergence. For
 	the second term, using $|\nabla\eta_k|\le C/k$ and
 	\eqref{eq:psi-decay}, we have
 	\begin{align*}
 		\norm{\psi_R\nabla\eta_k}_2^2
 		&\le\frac{C}{k^2}
 		\int_{\{k\le|x|\le2k\}}\psi_R(x)^2\,dx\le C_R k^{-n+4s-2}\to 0,
 	\end{align*}
 	since $-n+4s-2<0$ for $n\ge3$ and $s<1$. This proves the first convergence
 	in~\eqref{eq:cutoff-convergence}.
 	
 	We next consider the fractional seminorm. Set
 	\begin{equation*}
 		g_k:=(1-\eta_k)\psi_R=\psi_R-\varphi_k.
 	\end{equation*}
 	For every $x,y\in\R^n$, we have
 	\[
 	g_k(x)-g_k(y)
 	=(1-\eta_k(x))(\psi_R(x)-\psi_R(y))
 	+\psi_R(y)(\eta_k(y)-\eta_k(x)).
 	\]
 	Hence
 	\begin{equation}\label{eq:gk-seminorm-split}
 		[g_k]_s^2\le2I_k+2J_k,
 	\end{equation}
 	where
 	\begin{align*}
 		I_k&:=\iint_{\R^{2n}}
 		\frac{(1-\eta_k(x))^2|\psi_R(x)-\psi_R(y)|^2}
 		{|x-y|^{n+2s}}\,dx\,dy,\\
 		J_k&:=\int_{\R^n}\psi_R(y)^2\Theta_k(y)\,dy,
 	\end{align*}
 	and
 	\begin{equation}\label{eq:Xik}
 		\Theta_k(y):=\int_{\R^n}
 		\frac{|\eta_k(x)-\eta_k(y)|^2}{|x-y|^{n+2s}}\,dx.
 	\end{equation}
 	Since $[\psi_R]_s<+\infty$,
 	$0\le(1-\eta_k)^2\le1$, and $1-\eta_k(x)\to0$ for every fixed $x$, the
 	dominated convergence theorem gives
 	\begin{equation}\label{eq:Ik-zero}
 		I_k\to 0.
 	\end{equation}
 	To estimate $J_k$, we split the $y$-integration as
\[
 \begin{split}
 J_k
 &=\int_{\{|y|\le3k\}}\psi_R(y)^2\Theta_k(y)\,dy
   +\int_{\{|y|>3k\}}\psi_R(y)^2\Theta_k(y)\,dy=J_{1,k}+J_{2,k},
 \end{split}
\]
and we study the two terms separately.
 	\medskip
 	
 	\noindent -\,\,\emph{Study of $J_{1,k}$}.
 	We first note that, since $0\leq \eta_k\leq 1$ and $|\nabla\eta_k|\leq C/k$, we have
 	\[
 	|\eta_k(x)-\eta_k(y)|
 	\le C\min\left\{1,\frac{|x-y|}{k}\right\}.
 	\]
 	Thus, we obtain the uniform estimate
 	\begin{equation}\label{eq:Xik-uniform}
 		\Theta_k(y)\le Ck^{-2s}
 		\qquad\text{for every }y\in\R^n.
 	\end{equation}
 	Consequently,
 	\begin{equation}\label{eq:Jk-near}
 		J_{1,k} = \int_{\{|y|\le3k\}}\psi_R(y)^2\Theta_k(y)\,dy
 		\le Ck^{-2s}\int_{\{|y|\le3k\}}\psi_R(y)^2\,dy.
 	\end{equation}
 	By~\eqref{eq:psi-decay},
 	\begin{equation*}
 		\int_{\{|y|\le3k\}}\psi_R(y)^2\,dy
 		\le C_R
 		\begin{cases}
 			1,&n>4s,\\
 			1+\log k,&n=4s,\\
 			1+k^{4s-n},&n<4s.
 		\end{cases}
 	\end{equation*}
 	Since $n>2s$, we see that $J_{1,k}\to 0$ in
 	all three cases.
 	\medskip
 	
 	\noindent -\,\,\emph{Study of $J_{2,k}$}.  First of all we observe that, on the region $\{|y|>3k\}$, we have $\eta_k(y)=0$, while
 	$\eta_k(x)\ne0$ only if $|x|<2k$; as a consequence $|x-y|\ge |y|/3$ whenever the
 	integrand in~\eqref{eq:Xik} is nonzero, and 
 	we obtain the following
 	\emph{stronger version}
 	of \eqref{eq:Xik-uniform}:
 	\begin{equation*}
 		\begin{split}
 			\Theta_k(y)& =
 			\int_{\{|x|<2k\}}\frac{\eta_k(x)^2}{|x-y|^{n+2s}}\,dx
 			\leq
 			C|y|^{-n-2s}\int_{\{|x|<2k\}}\eta_k(x)^2\,dx \\
 			& \leq  Ck^n|y|^{-n-2s}
 			\qquad\text{for all }|y|>3k.
 		\end{split}
 	\end{equation*}
 	Using~\eqref{eq:psi-decay}, we conclude that
 	\begin{align}
 		J_{2,k} = \int_{\{|y|>3k\}}\psi_R(y)^2\Theta_k(y)\,dy
 		&\le C_Rk^n\int_{3k}^{+\infty}r^{-2n+2s-1}\,dr\le C_Rk^{-n+2s}\to 0,
 		\label{eq:Jk-far}
 	\end{align}
 	again because $n>2s$. Combining
 	\eqref{eq:Jk-near}--\eqref{eq:Jk-far}, we obtain $J_k\to0$. 
 	\vspace{0.1cm}
 	
 	Together with
 	\eqref{eq:gk-seminorm-split} and~\eqref{eq:Ik-zero}, this gives
 	$[g_k]_s\to0$, which is precisely the second convergence in
 	\eqref{eq:cutoff-convergence}, and the proof is complete.
 \end{proof}
 
 We are now in a position to prove the Liouville property stated in Theorem \ref{thm:intro-whole-space}. We split it, proving first the following
 
 \begin{theorem}\label{thm:whole-space-liouville}
 	Let $n\ge3$, $s\in(0,1)$ and $\gamma>0$. If $u\in H^1(\R^n)$ is a
 	nonnegative weak su\-per\-so\-lution of~\eqref{eq:whole-space-supersolution}, then
 	\begin{equation*}
 		u\equiv0.
 	\end{equation*}
 \end{theorem}
 
 \begin{proof}
 	Assume by contradiction that $u\not\equiv0$. Fix $R\ge R_0$, with $R_0$ as
 	in~\eqref{eq:R0-choice}, so that~\eqref{eq:barrier-negative} holds. Let
 	$(\varphi_k)_k$ be the nonnegative compactly supported sequence constructed in
 	Lemma~\ref{lem:whole-space-cutoff}. Since $u$ is a weak supersolution, we have
 	\begin{equation}\label{eq:supersolution-phik}
 		\int_{\R^n}\nabla u\cdot\nabla\varphi_k\,dx
 		-\gamma\Es(u,\varphi_k)\ge0
 		\qquad\text{for every }k\in\mathbb{N}.
 	\end{equation}
 	The convergences in~\eqref{eq:cutoff-convergence}, together with the
 	Cauchy--Schwarz inequality, allow us to pass to the limit in
 	\eqref{eq:supersolution-phik}, obtaining 
 	\begin{equation}\label{eq:supersolution-psiR}
 		\int_{\R^n}\nabla u\cdot\nabla\psi_R\,dx
 		-\gamma\Es(u,\psi_R)\ge0.
 	\end{equation}
 	We next move the operator onto $\psi_R$. To this end we first observe that, from
 	\eqref{eq:laplacian-special-psi}, \eqref{eq:fractional-bubble} and the scaling
 	relations~\eqref{eq:whole-space-scaling}, we have
 	\begin{equation}\label{eq:operator-psiR-L2}
 		-\Delta\psi_R\in L^2(\R^n),
 		\qquad
 		(-\Delta)^s\psi_R\in L^2(\R^n).
 	\end{equation}
 	Indeed, at infinity these functions decay respectively as
 	$|x|^{-n-2+2s}$ and $|x|^{-n-2s}$.
 	
 	Take $u_j\in C_0^\infty(\R^n)$ such that
 	\begin{equation}\label{eq:uj-density}
 		u_j\to  u
 		\qquad\text{strongly in }H^1(\R^n).
 	\end{equation}
 	Since $u_j$ is compactly supported, classical integration by parts gives
 	\[
 	\int_{\R^n}\nabla u_j\cdot\nabla\psi_R\,dx
 	=\int_{\R^n}u_j(-\Delta\psi_R)\,dx.
 	\]
 	Using~\eqref{eq:operator-psiR-L2} and~\eqref{eq:uj-density}, we may pass to
 	the limit and obtain
 	\begin{equation}\label{eq:local-duality-psiR}
 		\int_{\R^n}\nabla u\cdot\nabla\psi_R\,dx
 		=\int_{\R^n}u(-\Delta\psi_R)\,dx.
 	\end{equation}
 	The same argument applies to the fractional part. Indeed,
 	\eqref{eq:uj-density} and Lemma~\ref{lem:lower-order} imply
 	$u_j\to u$ strongly in $H^s(\R^n)$. Hence
 	\begin{equation*}
 		|\Es(u_j-u,\psi_R)|
 		\le [u_j-u]_s[\psi_R]_s\longrightarrow0.
 	\end{equation*}
 	On the other hand, by the fractional integration-by-parts formula, with the
 	normalization associated with~\eqref{eq:Es},
 	\[
 	\Es(u_j,\psi_R)
 	=\int_{\R^n}u_j(-\Delta)^s\psi_R\,dx.
 	\]
 	Using again~\eqref{eq:operator-psiR-L2} and passing to the limit gives
 	\begin{equation}\label{eq:fractional-duality-psiR}
 		\Es(u,\psi_R)
 		=\int_{\R^n}u(-\Delta)^s\psi_R\,dx.
 	\end{equation}
 	Combining~\eqref{eq:local-duality-psiR} and
 	\eqref{eq:fractional-duality-psiR}, we therefore find
 	\begin{align}
 		\int_{\R^n}\nabla u\cdot\nabla\psi_R\,dx
 		-\gamma\Es(u,\psi_R)
 		&=\int_{\R^n}u
 		\left[-\Delta\psi_R-\gamma(-\Delta)^s\psi_R\right]dx\notag\\
 		&=\int_{\R^n}u\,\mathcal L_{\gamma,s}\psi_R\,dx.
 		\label{eq:operator-duality-psiR}
 	\end{align}
 	The last integral is finite by~\eqref{eq:operator-psiR-L2} and $u\in L^2(\R^n)$.
 	
 	Finally, $u\ge0$ and $u\not\equiv0$, while
 	$\mathcal L_{\gamma,s}\psi_R<0$ everywhere by
 	\eqref{eq:barrier-negative}. Hence
 	\begin{equation*}
 		\int_{\R^n}u\,\mathcal L_{\gamma,s}\psi_R\,dx<0,
 	\end{equation*}
 	which contradicts~\eqref{eq:supersolution-psiR} and
 	\eqref{eq:operator-duality-psiR}. Thus $u\equiv0$. This closes the proof.
 \end{proof}
 
 The preceding theorem immediately gives the desired nonexistence result for
 power nonlinearities, without any restriction on the exponent.
 
 \begin{corollary}\label{cor:whole-space-power}
 	Let $r>1$. Then the problem
 	\begin{equation*}
 		-\Delta u-\gamma(-\Delta)^s u=u^{r-1}
 		\qquad\text{in }\R^n
 	\end{equation*}
 	admits no nontrivial nonnegative weak solution $u\in H^1(\R^n)$. Here, for
 	$r$ outside the Sobolev range, by weak solution we mean that
 	$u^{r-1}\in L^1_{\rm loc}(\R^n)$ and that
 	\begin{equation}\label{eq:whole-space-power-weak}
 		\int_{\R^n}\nabla u\cdot\nabla\varphi\,dx
 		-\gamma\Es(u,\varphi)
 		=\int_{\R^n}u^{r-1}\varphi\,dx
 	\end{equation}
 	holds for every $\varphi\in C_0^\infty(\R^n)$.
 \end{corollary}
 
 \begin{proof}
 	If $u\ge0$ satisfies~\eqref{eq:whole-space-power-weak}, then for every
 	nonnegative $\varphi\in C_0^\infty(\R^n)$,
 	\[
 	\int_{\R^n}\nabla u\cdot\nabla\varphi\,dx
 	-\gamma\Es(u,\varphi)
 	=\int_{\R^n}u^{r-1}\varphi\,dx\ge0.
 	\]
 	Thus $u$ is a nonnegative weak supersolution of
 	\eqref{eq:whole-space-supersolution}, and Theorem~\ref{thm:whole-space-liouville}
 	gives $u\equiv0$.
 \end{proof}
 
 \begin{remark}\label{rem:whole-space-general-f}
 	The power structure of the right-hand side is not used in the proof. The same
 	argument applies to any equation
 	\[
 	-\Delta u-\gamma(-\Delta)^s u=f(x,u)
 	\qquad\text{in }\R^n
 	\]
 	for which a nonnegative weak solution satisfies $f(x,u)\ge0$ a.e. Hence the
 	Liouville property is a consequence of the wrong-sign operator itself, rather
 	than of the particular growth of the nonlinearity.
 \end{remark}

  \section{Proof of Theorem \ref{thm:intro-threshold}}\label{sec:Proof_1.3}
In this section we prove Theorem \ref{thm:intro-threshold}. We first establish the attainment of the constant
$S_r(\gamma)$ in the subcritical range. We then study the normalized minimizers as
$\gamma\to C_{\rm emb}^-$, identify their limits as generalized eigenfunctions associated with
$C_{\rm emb}$, and
finally derive the uniform finite $L^R$-estimates needed in the sequel. 

Recall that
  \begin{equation*}
  	Q_{\gamma}(u) := \int_{\Omega}|\nabla u|^{2} \, dx - \gamma \, [u]^{2}_{s}, \quad u \in \mathcal{X}^{1,2}(\Omega),
  \end{equation*}
  \noindent and notice that, by the definition of $C_{\rm emb}$ in \eqref{eq:Cemb},
  	\begin{equation}\label{eq:coercivity}
  		Q_\gamma(u)\ge
  		\left(1-\frac\gamma{\Cemb}\right)
  		\norm{\nabla u}_{L^2(\Omega)}^2.
  \end{equation}

  \begin{proposition}\label{prop:attainment}
  	Let $1<r\le2^*$ and $0<\gamma<\Cemb$.
  	
  	If $r<2^*$, the infimum in~\eqref{eq:Sr} is achieved for every
  	$\gamma\in(0,\Cemb)$.
  	
  	If $r=2^*$, the attainment result is the one established in~\cite[Theorem~1.1]{BDVV6}:
  	if $n\ge5$, the infimum is achieved for every $\gamma\in(0,\Cemb)$, while if
  	$n=3,4$, there exists $\gamma^*\in[0,\Cemb)$ such that it is achieved for
  	every $\gamma\in(\gamma^*,\Cemb)$.
  	
  	In either case, every normalized minimizer $u_{\gamma,r}$ satisfies
  	\begin{equation}\label{eq:EL}
  		\int_\Omega\nabla u_{\gamma,r}\cdot\nabla\varphi\,dx
  		-\gamma\Es(u_{\gamma,r},\varphi)
  		=\Sr(\gamma)\int_\Omega
  		|u_{\gamma,r}|^{r-2}u_{\gamma,r}\varphi\,dx
  	\end{equation}
  	for every $\varphi\in \mathcal{X}^{1,2}(\Omega)$.
  \end{proposition}
  
  \begin{proof}
  	In view of Theorem \ref{thm:Da_BLMS}, only the case $r<2^*$ needs to be discussed. Let $(u_j)$ be a minimizing
  	sequence normalized by $\norm{u_j}_r=1$. By~\eqref{eq:coercivity}, $(u_j)$ is
  	bounded in $\mathcal{X}^{1,2}(\Omega)$. Hence, up to a subsequence,
  	\[
  	u_j\rightharpoonup u\quad\text{in }\mathcal{X}^{1,2}(\Omega),
  	\qquad
  	u_j\to u\quad\text{in }L^r(\Omega)\cap H^s(\R^n).
  	\]
  	Therefore $\norm{u}_r=1$, and weak lower semicontinuity of the Dirichlet
  	integral, together with the strong convergence in $H^s(\R^n)$, gives
  	\[
  	Q_\gamma(u)\le\liminf_{j\to\infty}Q_\gamma(u_j)=\Sr(\gamma).
  	\]
  	Thus, setting $u_{\gamma,r}:=u$, we obtain a normalized minimizer of
  	\eqref{eq:Sr}. The Euler equation follows from the Lagrange multiplier rule.
  	Testing it with $u_{\gamma,r}$ and using
  	$\norm{u_{\gamma,r}}_r=1$ shows that the multiplier is precisely
  	$\Sr(\gamma)$. This closes the proof.
  \end{proof}

  \medskip
  
  For the rest of the paper, whenever $r=2^*$ we restrict $\gamma$ to the range
  where the minimum in~\eqref{eq:Sr} is achieved, according to
  Proposition~\ref{prop:attainment}. This is no restriction when studying the
  limit 
  $$\gamma\to\Cemb^-.$$
  Let $\Phi_r$ be any nonzero element of the eigenspace associated with
  $\Cemb$, normalized so that
  \[
  \norm{\Phi_r}_{L^r(\Omega)}=1.
  \]
  In particular,
  \[
  \norm{\nabla\Phi_r}_{L^2(\Omega)}^2
  =\Cemb[\Phi_r]_s^2.
  \]
  \begin{proposition}\label{prop:basic-bounds}
  	For every $1<r\le2^*$ one has
  	\begin{equation}\label{eq:Sr-zero}
  		\Sr(\gamma)\longrightarrow0
  		\qquad\text{as }\gamma\to\Cemb^-.
  	\end{equation}
  	More precisely,
  	\begin{equation}\label{eq:Sr-upper}
  		0<\Sr(\gamma)\le(\Cemb-\gamma)[\Phi_r]_s^2.
  	\end{equation}
  	Moreover, every normalized minimizer $u_{\gamma,r}$ satisfies
  	\begin{equation}\label{eq:uniform-Hs}
  		[u_{\gamma,r}]_s^2\le[\Phi_r]_s^2
  	\end{equation}
  	and
  	\begin{equation*}
  		\norm{\nabla u_{\gamma,r}}_{L^2(\Omega)}^2
  		\le\Cemb[\Phi_r]_s^2.
  	\end{equation*}
  \end{proposition}
  
  \begin{proof}
  	Strict positivity of $\Sr(\gamma)$ follows from~\eqref{eq:coercivity} and the
  	Sobolev--Poincar\'e inequality. Since
  	$\norm{\nabla\Phi_r}_2^2=\Cemb[\Phi_r]_s^2$, the definition of
  	$\Sr(\gamma)$ gives
  	\[
  	\Sr(\gamma)\le Q_\gamma(\Phi_r)
  	=(\Cemb-\gamma)[\Phi_r]_s^2,
  	\]
  	which proves~\eqref{eq:Sr-upper} and~\eqref{eq:Sr-zero}.
  	
  	Set
  	\[
  	Q_{\Cemb}(u):=\norm{\nabla u}_2^2-\Cemb[u]_s^2\ge0.
  	\]
  	Since
  	\[
  	Q_\gamma(u)=Q_{\Cemb}(u)+(\Cemb-\gamma)[u]_s^2,
  	\]
  	we obtain
  	\[
  	(\Cemb-\gamma)[u_{\gamma,r}]_s^2
  	\le Q_\gamma(u_{\gamma,r})=\Sr(\gamma)
  	\le(\Cemb-\gamma)[\Phi_r]_s^2,
  	\]
  	which gives~\eqref{eq:uniform-Hs}. Finally,
  	\[
  	\begin{split}
  		\norm{\nabla u_{\gamma,r}}_2^2
  		&=\Sr(\gamma)+\gamma[u_{\gamma,r}]_s^2\\
  		&\le(\Cemb-\gamma)[\Phi_r]_s^2+\gamma[\Phi_r]_s^2
  		=\Cemb[\Phi_r]_s^2.
  	\end{split}
  	\]
  	This closes the proof.
  \end{proof}
  
  \begin{theorem}\label{thm:H1-convergence}
  	Let $(\gamma_k)_k\subset(0,\Cemb)$ satisfy $\gamma_k\to\Cemb^-$, and let
  	$u_{\gamma_k,r}$ be normalized minimizers. Then there exists
  	$0\not\equiv u_\infty\in \mathcal{X}^{1,2}(\Omega)$ such that, up to a subsequence,
  	\begin{equation*}
  		u_{\gamma_k,r}\to u_\infty\qquad\text{strongly in }\mathcal{X}^{1,2}(\Omega),
  	\end{equation*}
  	and
  	\begin{equation*}
  		\norm{u_\infty}_{L^r(\Omega)}=1,
  		\qquad
  		\norm{\nabla u_\infty}_2^2=\Cemb[u_\infty]_s^2.
  	\end{equation*}
  	In particular, $u_\infty$ is an $L^r$-normalized generalized eigenfunction
associated with $C_{\rm emb}$. As a consequence, if $C_{\rm emb}$ is simple and $\Phi_s$
is any fixed $L^r$-normalized generalized eigenfunction associated with it,
then necessarily
  	\begin{equation}\label{eq:plus-minus}
  		u_\infty=\Phi_s\qquad\text{or}\qquad u_\infty=-\Phi_s.
  	\end{equation}
  \end{theorem}
  
  \begin{proof}
  	By Proposition~\ref{prop:basic-bounds}, $(u_{\gamma_k,r})_k$ is uniformly bounded in
  	$\mathcal{X}^{1,2}(\Omega)$. Hence, up to a subsequence,
  	\[
  	u_{\gamma_k,r}\rightharpoonup u_\infty\quad\text{in }\mathcal{X}^{1,2}(\Omega),
  	\qquad
  	u_{\gamma_k,r}\to u_\infty\quad\text{in }H^s(\R^n).
  	\]
  	In particular,
  	\begin{equation}\label{eq:seminorm-conv}
  		[u_{\gamma_k,r}]_s^2\to[u_\infty]_s^2.
  	\end{equation}
  	Moreover,
  	\begin{equation}\label{eq:QC-zero}
  		\begin{split}
  			Q_{\Cemb}(u_{\gamma_k,r})
  			&=Q_{\gamma_k}(u_{\gamma_k,r})-(\Cemb-\gamma_k)[u_{\gamma_k,r}]_s^2\\
  			&=\Sr(\gamma_k)-(\Cemb-\gamma_k)[u_{\gamma_k,r}]_s^2\longrightarrow0,
  		\end{split}
  	\end{equation}
  	where we used~\eqref{eq:Sr-zero} and~\eqref{eq:uniform-Hs}. By weak lower
  	semicontinuity and~\eqref{eq:seminorm-conv},
  	\[
  	Q_{\Cemb}(u_\infty)
  	\le\liminf_{k\to\infty}Q_{\Cemb}(u_{\gamma_k,r})=0.
  	\]
  	Since $Q_{\Cemb}\ge0$, it follows that
  	\begin{equation}\label{eq:limit-QC-zero}
  		Q_{\Cemb}(u_\infty)=0.
  	\end{equation}
  	Combining~\eqref{eq:QC-zero}, \eqref{eq:seminorm-conv} and
  	\eqref{eq:limit-QC-zero}, we obtain
  	\[
  	\norm{\nabla u_{\gamma_k,r}}_2^2
  	=Q_{\Cemb}(u_{\gamma_k,r})+\Cemb[u_{\gamma_k,r}]_s^2
  	\longrightarrow\Cemb[u_\infty]_s^2
  	=\norm{\nabla u_\infty}_2^2.
  	\]
  	Weak convergence together with convergence of the norms gives strong
  	convergence in $H^1_0(\Omega)$ and hence, together with
  	\eqref{eq:seminorm-conv}, in $\mathcal{X}^{1,2}(\Omega)$. In particular,
  	$$u_{\gamma_k,r}\to u_\infty$$ 
  	strongly in $L^r(\Omega)$, including when $r=2^*$.
  	Therefore $\|u_\infty\|_r=1$, and \eqref{eq:limit-QC-zero} identifies $u_\infty$ as a generalized
eigenfunction associated with $C_{\rm emb}$. 
The last assertion follows
under the additional assumption that $C_{\rm emb}$ is simple.
 
 This closes the proof.
  \end{proof}
  
  We now prove estimates which are uniform as $\gamma\to C_{\rm emb}^-$.
Our use of truncated-power test functions is inspired by the
Moser-type arguments for mixed local--nonlocal equations in
\cite{SVWZ, SVWZ2}; here, the main difficulty is the unfavorable sign
of the fractional term in \eqref{eq:EL}.
\vspace{0.1cm}

  For $\beta>1$ and $T\ge1$, define
  \begin{equation}\label{eq:GTHT}
  	G_{\beta,T}(t):=t\min\{|t|,T\}^{\beta-1},
  	\qquad
  	H_{\beta,T}(t):=t\min\{|t|,T\}^{2\beta-2}.
  \end{equation}
  
  \begin{lemma}\label{lem:algebra}
  	For every $\beta>1$ and $T\ge1$, the functions in~\eqref{eq:GTHT} satisfy
  	\begin{equation}\label{eq:derivative-ineq}
  		H_{\beta,T}'(t)\ge c_\beta|G_{\beta,T}'(t)|^2
  		\quad\text{for a.e. }t\in\R,
  		\qquad
  		c_\beta:=\frac{2\beta-1}{\beta^2},
  	\end{equation}
  	\begin{equation}\label{eq:nonlocal-algebra}
  		(a-b)\bigl(H_{\beta,T}(a)-H_{\beta,T}(b)\bigr)
  		\le2|G_{\beta,T}(a)-G_{\beta,T}(b)|^2
  		\qquad\text{for every }a,b\in\R,
  	\end{equation}
  	and
  	\begin{equation}\label{eq:uH=G2}
  		tH_{\beta,T}(t)=G_{\beta,T}(t)^2
  		\qquad\text{for every }t\in\R.
  	\end{equation}
  \end{lemma}
  
  \begin{proof}
  	Set $h(t):=\min\{|t|,T\}^{\beta-1}$, so that $G(t)=th(t)$ and
  	$H(t)=th(t)^2$. If $|t|<T$, then
  	\[
  	H'(t)=(2\beta-1)|t|^{2\beta-2},
  	\qquad
  	|G'(t)|^2=\beta^2|t|^{2\beta-2},
  	\]
  	while, if $|t|>T$,
  	\[
  	H'(t)=T^{2\beta-2},
  	\qquad
  	|G'(t)|^2=T^{2\beta-2}.
  	\]
  	Since $c_\beta\le1$, this proves~\eqref{eq:derivative-ineq};
  	\eqref{eq:uH=G2} is immediate.
  	
  	For~\eqref{eq:nonlocal-algebra}, a direct computation gives
  	\begin{equation}\label{eq:algebra-identity}
  		(a-b)(H(a)-H(b))
  		=(G(a)-G(b))^2-ab\,(h(a)-h(b))^2.
  	\end{equation}
  	If $ab\ge0$, the conclusion follows immediately. If $ab<0$, assume without
  	loss of generality that $|a|\ge|b|$. Since $h$ is nondecreasing as a function
  	of $|t|$,
  	\[
  	|ab|\,(h(a)-h(b))^2
  	\le |a||b|h(a)^2
  	\le |a|^2h(a)^2
  	\le |G(a)-G(b)|^2,
  	\]
  	because $G(a)$ and $G(b)$ have opposite signs. Formula~\eqref{eq:algebra-identity}
  	then yields~\eqref{eq:nonlocal-algebra}.
  \end{proof}

  \begin{proposition}\label{prop:finite-LR}
  	Let $R\in[1,+\infty)$. There exist $\gamma_R<\Cemb$ and $M_R>0$ such that
  	\begin{equation}\label{eq:uniform-LR}
  		\norm{u_{\gamma,r}}_{L^R(\Omega)}\le M_R
  	\end{equation}
  	for every $\gamma\in(\gamma_R,\Cemb)$ and every normalized minimizer
  	$u_{\gamma,r}$.
  \end{proposition}
  
  \begin{proof}
  	Fix $\beta>1$, to be chosen below, and, for $T\ge1$, set
  	\[
  	G_{\gamma,T}:=G_{\beta,T}(u_{\gamma,r}),\qquad
  	H_{\gamma,T}:=H_{\beta,T}(u_{\gamma,r}).
  	\]
  	Since $H_{\beta,T}$ is Lipschitz and vanishes at the origin,
  	$H_{\gamma,T}\in \mathcal{X}^{1,2}(\Omega)$ and can be used as a test function in~\eqref{eq:EL}.
  	By~\eqref{eq:derivative-ineq},
  	\begin{equation}\label{eq:local-test}
  		\int_\Omega\nabla u_{\gamma,r}\cdot\nabla H_{\gamma,T}\,dx
  		\ge c_\beta\norm{\nabla G_{\gamma,T}}_2^2.
  	\end{equation}
  	On the other hand,~\eqref{eq:nonlocal-algebra} gives
  	\begin{equation}\label{eq:nonlocal-test}
  		\Es(u_{\gamma,r},H_{\gamma,T})\le2[G_{\gamma,T}]_s^2.
  	\end{equation}
  	By Remark~\ref{rem:lower-order-X}, choosing
  	$\varepsilon=c_\beta/(4\Cemb)$ and using $\gamma<\Cemb$, we infer from
  	\eqref{eq:local-test}--\eqref{eq:nonlocal-test} that
  	\begin{equation}\label{eq:test-lower}
  		\int_\Omega\nabla u_{\gamma,r}\cdot\nabla H_{\gamma,T}\,dx-\gamma\Es(u_{\gamma,r},H_{\gamma,T})
  		\ge\frac{c_\beta}{2}\norm{\nabla G_{\gamma,T}}_2^2-C_\beta\norm{G_{\gamma,T}}_2^2,
  	\end{equation}
  	where $C_\beta>0$ is independent of $T$, $\gamma$ and the chosen minimizer $u_{\gamma,r}$.
  	Moreover, by~\eqref{eq:uH=G2}, the right-hand side of~\eqref{eq:EL} becomes
  	\begin{equation*}
  		\Sr(\gamma)\int_\Omega|u_{\gamma,r}|^{r-2}G_{\gamma,T}^2\,dx.
  	\end{equation*}
  	
  	We distinguish two cases. If $2\le r\le2^*$, H\"older's inequality and
  	$\norm{u_{\gamma,r}}_r=1$ give
  	\[
  	\int_\Omega|u_{\gamma,r}|^{r-2}G_{\gamma,T}^2\,dx\le\norm{G_{\gamma,T}}_r^2
  	\le C\norm{G_{\gamma,T}}_{2^*}^2.
  	\]
  	Using~\eqref{eq:Sobolev} in~\eqref{eq:test-lower} and recalling that
  	$\Sr(\gamma)\to0$, we can absorb the last term for $\gamma$ sufficiently
  	close to $\Cemb$. Hence
  	\begin{equation}\label{eq:key-G}
  		\norm{G_{\gamma,T}}_{2^*}^2\le K_\beta\bigl(1+\norm{G_{\gamma,T}}_2^2\bigr),
  	\end{equation}
  	with $K_\beta$ independent of $T$, $\gamma$ and the minimizer $u_{\gamma,r}$.
  	
  	If $1<r<2$, then pointwise
  	\begin{equation*}
  		|u_{\gamma,r}|^{r-2}G_{\gamma,T}^2\le1+G_{\gamma,T}^2.
  	\end{equation*}
  	Indeed, on $\{|u_{\gamma,r}|\le1\}$ the left-hand side is at most $1$, whereas on
  	$\{|u_{\gamma,r}|>1\}$ we have $|u_{\gamma,r}|^r\le|u_{\gamma,r}|^2$. Since $\Sr(\gamma)$ is bounded for
  	$\gamma$ close to $\Cemb$,~\eqref{eq:key-G} follows again from
  	\eqref{eq:test-lower} and~\eqref{eq:Sobolev}.
  	
  	We now remove the $L^2$-term in~\eqref{eq:key-G} without any iteration. Let
  	$P:=2^*$ and, for $M\ge1$, set
  	\[
  	E_M:=\{x\in\Omega:|u_{\gamma,r}(x)|>M\}.
  	\]
  	Since $\norm{u_{\gamma,r}}_r=1$, Chebyshev's inequality gives
  	\begin{equation*}
  		|E_M|\le M^{-r}.
  	\end{equation*}
  	On $\Omega\setminus E_M$ one has $|G_{\gamma,T}|\le|u_{\gamma,r}|^\beta\le M^\beta$, while on
  	$E_M$, H\"older's inequality gives
  	\[
  	\int_{E_M}|G_{\gamma,T}|^2\,dx
  	\le |E_M|^{1-2/P}\norm{G_{\gamma,T}}_P^2
  	\le M^{-2r/n}\norm{G_{\gamma,T}}_P^2.
  	\]
  	Consequently,
  	\begin{equation}\label{eq:G-L2-level}
  		\norm{G_{\gamma,T}}_2^2\le|\Omega|M^{2\beta}
  		+M^{-2r/n}\norm{G_{\gamma,T}}_P^2.
  	\end{equation}
  	Inserting~\eqref{eq:G-L2-level} into~\eqref{eq:key-G}, and choosing $M$ so
  	large that $K_\beta M^{-2r/n}\le1/2$, we obtain
  	\begin{equation*}
  		\norm{G_{\gamma,T}}_P^2\le A_\beta,
  	\end{equation*}
  	where $A_\beta$ is independent of $T$, $\gamma$ and the minimizer $u_{\gamma,r}$.
  	Letting $T\to+\infty$ and using monotone convergence,
  	\begin{equation}\label{eq:beta-P}
  		\norm{u_{\gamma,r}}_{L^{\beta P}(\Omega)}^{2\beta}\le A_\beta.
  	\end{equation}
  	Given $R<+\infty$, choose a fixed $\beta>1$ such that $R\le\beta P$.
  	Since $\Omega$ is bounded,~\eqref{eq:uniform-LR} follows from
  	\eqref{eq:beta-P}.
  \end{proof}
  
  \begin{corollary}\label{cor:rhs}
  	For every fixed $q\in[1,+\infty)$,
  	\begin{equation*}
  		\Sr(\gamma)|u_{\gamma,r}|^{r-2}u_{\gamma,r}
  		\longrightarrow0
  		\qquad\text{strongly in }L^q(\Omega)
  	\end{equation*}
  	as $\gamma\to\Cemb^-$, uniformly with respect to the normalized minimizers $u_{\gamma,r}$.
  \end{corollary}
  
  \begin{proof}
  	Set $R:=\max\{1,q(r-1)\}$. Proposition~\ref{prop:finite-LR} and H\"older's
  	inequality give
  	\[
  	\norm{\Sr(\gamma)|u_{\gamma,r}|^{r-2}u_{\gamma,r}}_q
  	\le \Sr(\gamma)|\Omega|^{1/q-(r-1)/R}\norm{u_{\gamma,r}}_R^{r-1}
  	\le C_R\Sr(\gamma),
  	\]
  	and the conclusion follows from~\eqref{eq:Sr-zero}.
  \end{proof}
  
  \begin{proof}[Proof of Theorem \ref{thm:intro-threshold}]
  	This is now a simple combination of  Proposition \ref{prop:basic-bounds}, Theorem \ref{thm:H1-convergence} and Proposition \ref{prop:finite-LR}. This closes the proof.
  \end{proof}

  \section{Proof of Theorem \ref{thm:intro-C1}}\label{sec:Proof_1.4}
 In this section we prove Theorem 1.4. Throughout the section, we assume that
\[
0<s<\frac12
\]
and that $\Omega$ is of class $C^{1,1}$.
  
  We shall use directly the global regularity estimate in~\cite[Theorem~1.8]{RE8}.
  There is only one normalization issue to keep track of. In~\cite{RE8}, the
  fractional Laplacian contains the standard positive factor $c_{n,s}$, whereas
  our definition associated with~\eqref{eq:seminorm} corresponds to
  $c_{n,s}=1$. Hence our equation
  \[
  -\Delta v-\gamma(-\Delta)^s v=g
  \]
  corresponds in the notation of~\cite{RE8} to the signed measure
  \begin{equation*}
  	\mu_\gamma=-\frac{\gamma}{c_{n,s}}\delta_s.
  \end{equation*}
  Since $s$ is fixed and $0\le\gamma\le\Cemb$, the total variations of these
  measures are uniformly bounded. Therefore, for every fixed
  $q\in[2,+\infty)$,~\cite[Theorem~1.8, formula~(1.26)]{RE8} yields
  \begin{equation}\label{eq:DPLSV-estimate}
  	\norm{v}_{W^{2,q}(\Omega)}
  	\le C\bigl(\norm{g}_{L^q(\Omega)}+\norm{v}_{L^2(\Omega)}\bigr),
  \end{equation}
  where $C$ is independent of $\gamma\in[0,\Cemb]$.
  
  \begin{proposition}\label{prop:uniform-W2q}
  	For every fixed $q\in[2,+\infty)$ there exist $\gamma_q<\Cemb$ and $C_q>0$
  	such that
  	\begin{equation}\label{eq:uniform-W2q}
  		\norm{u_{\gamma,r}}_{W^{2,q}(\Omega)}\le C_q
  	\end{equation}
  	for every $\gamma\in(\gamma_q,\Cemb)$ and every normalized minimizer $u_{\gamma,r}$.
  	In particular, if $q>n$ and $\alpha=1-n/q$, then
  	\begin{equation*}
  		\norm{u_{\gamma,r}}_{C^{1,\alpha}(\overline\Omega)}\le C_q.
  	\end{equation*}
  \end{proposition}
  
  \begin{proof}
  	Apply~\eqref{eq:DPLSV-estimate} to~\eqref{eq:EL}, with
  	\[
  	g=\Sr(\gamma)|u_{\gamma,r}|^{r-2}u_{\gamma,r}.
  	\]
  	The $L^q$-norm of $g$ is uniformly bounded, and actually tends to zero, by
  	Corollary~\ref{cor:rhs}, while the $L^2$-norm of $u_{\gamma,r}$ is uniformly
  	bounded by Proposition~\ref{prop:basic-bounds}. This proves
  	\eqref{eq:uniform-W2q}. The second assertion follows from the
  	Sobolev--Morrey embedding.
  \end{proof}
  
  \begin{proposition}\label{prop:strong-W2q}
  	Let $\gamma_k\to C_{\rm emb}^{-}$ and let $u_{\gamma_k,r}$ be normalized
minimizers. After passing to a subsequence and relabelling, let
$u_\infty$ be the limit supplied by Theorem~\ref{thm:H1-convergence}, so that
\[
 u_{\gamma_k,r}\longrightarrow u_\infty
 \qquad\text{strongly in }\mathcal{X}^{1,2}(\Omega).
\]
Then, for every fixed $q\in[2,+\infty)$,  	\begin{equation}\label{eq:strong-W2q}
  		u_{\gamma_k,r}\to u_\infty\qquad\text{strongly in }W^{2,q}(\Omega).
  	\end{equation}
  	In particular, if $q>n$,
  	\begin{equation}\label{eq:strong-C1}
  		u_{\gamma_k,r}\to u_\infty\qquad\text{strongly in }C^1(\overline\Omega).
  	\end{equation}
  	If $C_{\rm emb}$ is simple, then there exists $\sigma\in\{-1,1\}$ such that  	\begin{equation*}
  		u_{\gamma_k,r}\to\sigma\Phi_s
  		\qquad\text{strongly in }W^{2,q}(\Omega)\text{ and }C^1(\overline\Omega),
  	\end{equation*}
  	where $\Phi_s$ is any fixed
$L^r$-normalized generalized eigenfunction associated with $C_{\rm emb}$.
  \end{proposition}
  
  \begin{proof}
  	The limit $u_\infty$ satisfies
  	\[
  	-\Delta u_\infty-\Cemb(-\Delta)^s u_\infty=0.
  	\]
  	Applying~\eqref{eq:DPLSV-estimate} at $\gamma=\Cemb$ gives
  	$u_\infty\in W^{2,q}(\Omega)$. Hence, by the above equation,
  	\begin{equation}\label{eq:frac-limit-Lq}
  		(-\Delta)^s u_\infty
  		=\frac{-\Delta u_\infty}{\Cemb}\in L^q(\Omega).
  	\end{equation}
  	Set
  	\[
  	f_{\gamma_k}:=\Sr(\gamma_k)|u_{\gamma_k,r}|^{r-2}u_{\gamma_k,r},
  	\qquad
  	w_k:=u_{\gamma_k,r}-u_\infty.
  	\]
  	Then
  	\begin{equation*}
  		-\Delta w_k-\gamma_k(-\Delta)^s w_k
  		=f_{\gamma_k}-(\Cemb-\gamma_k)(-\Delta)^s u_\infty.
  	\end{equation*}
  	Using~\eqref{eq:DPLSV-estimate} once more, together with
  	Corollary~\ref{cor:rhs},~\eqref{eq:frac-limit-Lq}, and the strong
  	$\mathcal{X}^{1,2}$-convergence in Theorem~\ref{thm:H1-convergence}, we obtain
  	\[
  	\begin{split}
  		\norm{u_{\gamma_k,r}-u_\infty}_{W^{2,q}(\Omega)}
  		\le C\bigl(&\norm{f_{\gamma_k}}_q
  		+(\Cemb-\gamma_k)\norm{(-\Delta)^s u_\infty}_q\\
  		&+\norm{u_{\gamma_k,r}-u_\infty}_2\bigr)\longrightarrow0.
  	\end{split}
  	\]
  	This proves~\eqref{eq:strong-W2q};~\eqref{eq:strong-C1} follows for $q>n$.
  	The last assertion follows from~\eqref{eq:plus-minus}.
  \end{proof}
  
  \begin{proof}[Proof of Theorem \ref{thm:intro-C1}]
  	Combining Proposition \ref{prop:uniform-W2q} and Proposition \ref{prop:strong-W2q} we get the desired conclusion. This closes the proof.
  \end{proof}
  
  \section{Proof of Theorems \ref{thm:intro-positivity} and \ref{thm:intro-supercritical}}\label{sec:Proof_1.5}
  In this last section, we complete the analysis of the bounded-domain problem. We first prove
Theorem \ref{thm:intro-positivity} by combining the convergence obtained in Sections~
\ref{sec:Proof_1.3} and \ref{sec:Proof_1.4} with the
small-$s$ spectral information of Lemma~\ref{lem:DPLSV-small-s} and the Hopf lemma. We then turn
to the Sobolev supercritical regime and prove Theorem~
\ref{thm:intro-supercritical} by means of a
Pohozaev identity.
\medskip

 We begin with the positivity of the normalized minimizers.
 For this part of the section, we assume that $\Omega$ is connected and
of class $C^{1,1}$, and fix $1<r\le2^*$.

 Let $\varphi_1$
denote the positive first Dirichlet eigenfunction of $-\Delta$, normalized in such a way that
$\|\varphi_1\|_{L^r(\Omega)}=1$. By the Hopf lemma,
  \begin{equation*}
  	\varphi_1>0\quad\text{in }\Omega,
  	\qquad
  	\partial_{\nu_{\rm in}}\varphi_1>0\quad\text{on }\partial\Omega,
  \end{equation*}
  where $\nu_{\rm in}$ denotes the interior unit normal.
  
  \begin{lemma}\label{lem:C1-stability}
  	Let $\phi\in C^1(\overline\Omega)$ satisfy
  	\[
  	\phi=0\quad\text{on }\partial\Omega,
  	\qquad
  	\phi>0\quad\text{in }\Omega,
  	\qquad
  	\inf_{\partial\Omega}\partial_{\nu_{\rm in}}\phi>0.
  	\]
  	Then there exists $\eta>0$ such that every
  	$v\in C^1(\overline\Omega)$ satisfying $v=0$ on $\partial\Omega$ and
  	\[
  	\norm{v-\phi}_{C^1(\overline\Omega)}<\eta
  	\]
  	is strictly positive in $\Omega$.
  \end{lemma}
  
  \begin{proof}
  	Set
  	\[
  	a:=\min_{\partial\Omega}\partial_{\nu_{\rm in}}\phi>0.
  	\]
  	By continuity, there exists $\delta>0$ such that the derivative of $\phi$ in
  	the inward normal direction is at least $3a/4$ in the boundary layer
  	\[
  	\Omega_\delta:=\{x\in\Omega:\operatorname{dist}(x,\partial\Omega)<\delta\}.
  	\]
  	If $\norm{v-\phi}_{C^1}<a/4$, the corresponding derivative of $v$ is at
  	least $a/2$ in a possibly smaller boundary layer. Since $v=0$ on
  	$\partial\Omega$, integration along inward normal segments gives
  	\[
  	v(x)\ge\frac a2\operatorname{dist}(x,\partial\Omega)>0
  	\qquad\text{for }x\in\Omega_\delta.
  	\]
  	On the compact set
  	$K_\delta:=\{x\in\Omega:\operatorname{dist}(x,\partial\Omega)\ge\delta\}$,
  	we have $m_\delta:=\min_{K_\delta}\phi>0$. Requiring in addition
  	$\norm{v-\phi}_{C^0}<m_\delta/2$ yields $v>0$ also on $K_\delta$.
  \end{proof}
  Thanks to the above lemma, we can now provide the
    \begin{proof}[Proof of Theorem \ref{thm:intro-positivity}]
  	Let $\eta>0$ be given by Lemma~\ref{lem:C1-stability} with
  	$\phi=\varphi_1$. By Lemma~\ref{lem:DPLSV-small-s}, reducing $s_0$ if necessary,
  	we can find $s_1\in(0,s_0)$ such that
  	\begin{equation*}
  		\norm{\Phi_s-\varphi_1}_{C^1(\overline\Omega)}<\frac\eta2
  		\qquad\text{for every }s\in(0,s_1),
  	\end{equation*}
  	where $\Phi_s$ is the positive $L^r$-normalized generalized eigenfunction
associated with $C_{\rm emb}$.  	
  	We now fix $s\in(0,s_1)$, and  we argue by contradiction. If the assertion were false, there would exist a sequence
$\gamma_k\to C_{\rm emb}^{-}$ and normalized minimizers $u_{\gamma_k,r}$
that are neither strictly positive nor strictly negative in $\Omega$.  	

By Theorem~4.3 and the simplicity of $C_{\rm emb}$, after passing
to a subsequence there exists $\sigma\in\{-1,1\}$ such that  	\[
  	u_{\gamma_k,r}\to\sigma\Phi_s\qquad\text{strongly in }\mathcal{X}^{1,2}(\Omega).
  	\]
  	Proposition~\ref{prop:strong-W2q}, with any $q>n$, improves this convergence to
  	\begin{equation*}
  		u_{\gamma_k,r}\to\sigma\Phi_s
  		\qquad\text{strongly in }C^1(\overline\Omega).
  	\end{equation*}
  	Therefore, for $k$ sufficiently large,
  	\[
  	\norm{\sigma u_{\gamma_k,r}-\varphi_1}_{C^1(\overline\Omega)}
  	\le\norm{\sigma u_{\gamma_k,r}-\Phi_s}_{C^1(\overline\Omega)}
  	+\norm{\Phi_s-\varphi_1}_{C^1(\overline\Omega)}<\eta.
  	\]
  	Lemma~\ref{lem:C1-stability} gives $\sigma u_{\gamma_k,r}>0$ in $\Omega$, contradicting
our choice of $u_{\gamma_k,r}$. This proves the existence of
  	$\overline\gamma_{r,s}$ and hence the theorem.
  \end{proof}
  
%  \begin{corollary}\label{cor:critical}
%  	Assume the hypotheses of Theorem~\ref{thm:positivity},
%  	\textcolor{red}{and let
%$1<r\le2^*$, $r\neq2$}. For
%  	all $s>0$ sufficiently small and $\gamma<\Cemb(s)$ sufficiently close to
%  	$\Cemb(s)$, the problem
%  	\begin{equation}\label{eq:critical-problem}
%  		\begin{cases}
%  			-\Delta v - \gamma(-\Delta)^s v = |v|^{r-2}v&\text{in }\Omega,\\
%  			v=0&\text{in }\R^n\setminus\Omega,
%  		\end{cases}
%  	\end{equation}
%  	admits a positive weak solution.
%  \end{corollary}
%   \begin{proof}
%  	%By~\cite[Theorem~1.1]{BDVV6}, 
%  	We know that for $\gamma$ sufficiently close to $\Cemb(s)$
%  	the minimum $S_{r}(\gamma)$ is attained. Let $u_{\gamma,r}$ be the
%  	positive normalized minimizer given by Theorem~\ref{thm:positivity}. Then
%  	\[
%  	v_\gamma:=S_{r}(\gamma)^{1/(r-2)}u_{\gamma,r}
%  	\]
%  	is a positive weak solution of~\eqref{eq:critical-problem}. This closes the proof.
%  \end{proof}
We now turn to establish Theorem \ref{thm:intro-supercritical}. The proof is based on a Pohozaev-type identity adapted to the wrong-sign
operator. We first recall 
the identity in the form needed below. 
  \begin{proposition}\label{prop:Poho_Limitati}
  Let $\Omega \subset \mathbb{R}^n$ be an open, bounded and connected set with $C^2$-smooth boundary $\partial \Omega$. Let us denote by $\nu$ the outer unit normal to $\partial \Omega$. Let 
  $$f:\mathbb{R}\to \mathbb{R}$$ be a continuous function.
  Let $s \in (0,1/2)$.
  Let $u \in C^{2}(\overline{\Omega})\cap C(\mathbb{R}^n)$  be a solution to
  \begin{equation}\label{eq:Per_Poho_Limitati}
  	\begin{cases}
  		-\Delta u - \gamma(-\Delta)^s u =f(u)&\text{in }\Omega,\\
  		u=0&\text{in }\R^n\setminus\Omega.
  	\end{cases}
  \end{equation}
   Then the following identity holds true
  \begin{equation}\label{eq:Poho_Limitati}
  		\dfrac{n-2}{2} \int_{\Omega}|\nabla u|^2 - \gamma \, \dfrac{n-2s}{2} \, [u]^2_s - n \int_{\Omega}F(u) =-\dfrac{1}{2}\int_{\de \Omega}(x\cdot \nu) \, \left(\dfrac{\partial u}{\partial \nu}\right)^2\,d\sigma,
  \end{equation}
  \noindent where $F(t):=\int_{0}^{t}f(\tau)\, d\tau$.
  \end{proposition}
The identity is the same as the one proved 
in \cite[Theorem~1.1]{ABiswas}. We only
need to keep track of two minor points. First, our fractional Laplacian differs
from the normalization used in \cite{ABiswas} by a positive multiplicative constant,
which is absorbed into $\gamma$. Second, although
\cite{ABiswas} states the equation
with a nonnegative coefficient in front of the fractional term, the proof of
the Pohozaev identity treats the local and nonlocal contributions separately,
and the sign of this coefficient does not enter the computation once the
required regularity is available. In the present formulation this regularity is
immediate: since $u\in C^2(\overline\Omega)$ and $u=0$ in
$\mathbb R^n\setminus\Omega$, its zero extension is globally Lipschitz, while
for $s<1/2$ the fractional term is bounded up to the boundary. Therefore, after adjusting the coefficient to account for the different
normalization, the computation in \cite{ABiswas} applies with a negative fractional
coefficient and yields the identity above. 
\begin{proof}[Proof of Theorem \ref{thm:intro-supercritical}]
 Up to a translation of coordinates, we may assume that $\Omega$ is
star-shaped with respect to the origin, so that
$x\cdot\nu(x)\ge0$ for every $x\in\partial\Omega$.

 	Let us start noticing that, due to the assumption $p>\tfrac{n+2}{n-2}$, the following hold true:
 	\begin{equation}\label{eq:Coeff_positivo}
 		\dfrac{n-2}{2}-\dfrac{n}{p+1}>0,
 	\end{equation}
 	\noindent and 
 	\begin{equation}\label{eq:Coeff_positivo2}
 		\dfrac{n-2s}{2}-\dfrac{n}{p+1}>0.
 	\end{equation}
 	Assume now that there exists a non-trivial solution to \eqref{eq:Non_Esistenza_Limitati} and let us test \eqref{eq:Non_Esistenza_Limitati} with $u$ itself: this gives
 	\begin{equation}\label{eq:Testando_con_u_sopracritico}
 		\int_{\Omega}|\nabla u|^2 -\gamma \, [u]^{2}_{s} = \int_{\Omega}|u|^{p+1}.
 	\end{equation}
 	Plugging the latter into \eqref{eq:Poho_Limitati}, we get
 	\begin{equation*}
 		\begin{aligned}
 			\dfrac{n-2}{2} \int_{\Omega}|\nabla u|^2 &- \gamma \, \dfrac{n-2s}{2} \, [u]^2_s - \dfrac{n}{p+1}\int_{\Omega}|\nabla u|^2 + \dfrac{\gamma \, n}{p+1}[u]^{2}_{s} \\
 			& =-\dfrac{1}{2}\int_{\de \Omega}(x\cdot \nu) \, \left(\dfrac{\partial u}{\partial \nu}\right)^2\,d\sigma,
 		\end{aligned}
 	\end{equation*}
 	Using the assumption on $\gamma$, the fact that $[u]_s>0$ for a nontrivial
solution, and the definition of $C_{\rm emb}$ in \eqref{eq:Cemb}, we obtain 	\begin{equation}
 		\begin{aligned}
0
&<\left[
 C_{\rm emb}\left(\frac{n-2}{2}-\frac{n}{p+1}\right)
 -\gamma\left(\frac{n-2s}{2}-\frac{n}{p+1}\right)
 \right][u]_s^2\\
&\le
 \left(\frac{n-2}{2}-\frac{n}{p+1}\right)
 \int_\Omega |\nabla u|^2\,dx
 -\gamma\left(\frac{n-2s}{2}-\frac{n}{p+1}\right)[u]_s^2\\
&=-\frac12\int_{\partial\Omega}(x\cdot\nu)
 \left(\frac{\partial u}{\partial\nu}\right)^2\,d\sigma
 \le0.
\end{aligned}
 	\end{equation}
 	This contradiction proves the theorem.
 	 \end{proof}


\begin{thebibliography}{99}
 
\bibitem{AC}
N. Abatangelo, M. Cozzi,
{\em An elliptic boundary value problem with fractional nonlinearity},
SIAM J. Math. Anal. {\bf 53}(3), (2021), 3577--3601. 
%
%\bibitem{ALI20}
%	N. Alibaud, F. del Teso, J. Endal, E. R. Jakobsen, \emph{The Liouville theorem and linear operators satisfying the maximum principle},
%	J. Math. Pures Appl. (9) {\bf 142}, (2020), 229--242.


\bibitem{AntoCozzi}
C.A. Antonini, M. Cozzi,
{\em Global gradient regularity and a Hopf lemma for quasilinear operators of mixed local-nonlocal type},
J. Differential Equations {\bf 425}, (2025), 342--382. 
%\url{https://doi.org/10.1016/j.jde.2025.01.030}

\bibitem{BaDPQ}
B. Barrios, L.M. Del Pezzo, A. Quaas,
{\em Mixed local and nonlocal laplacian without standard critical exponent for Lane-Emden equation}, preprint. \url{https://arxiv.org/abs/2507.12258}

%\bibitem{RE11}
%S. Bhowmick, S. Ghosh, V. Kumar, \emph{Superlinear problems involving nonlinear superposition operators of mixed fractional order}, 
%Proc. Roy. Soc. Edinburgh Sect. A. Published online 2026: 1--26. doi:10.1017/prm.2026.10124.


\bibitem{RE12}
S. Bhowmick, S. Ghosh, V. Kumar, \emph{Infinitely many solutions for nonlinear
superposition operators of mixed fractional order involving critical exponent}, 
Discrete Contin. Dyn. Syst. Ser. S {\bf 31}, (2027), 264--287.


%\bibitem{RE13}
%S. Bhowmick, S. Ghosh, V. Kumar, R. Lakshmi, \emph{Regularity of superposition operators
%of mixed fractional order}, preprint, \url{https://arxiv.org/abs/2606.01449}

\bibitem{BDVV-CPDE22}
S. Biagi, S. Dipierro, E. Valdinoci, E. Vecchi,
{\em Mixed local and nonlocal elliptic operators: regularity and maximum principles}, 
Comm. Partial Differential Equations {\bf 47} (3) (2022), 585--629.


%\bibitem{BDVV2-2021}
%S. Biagi, S. Dipierro, E. Valdinoci, E. Vecchi,
%{\em Semilinear elliptic equations involving mixed local
%and nonlocal operators},
% Proc. Roy. Soc. Edinburgh Sect. A {\bf 151}(5), (2021), 1611--1641. 


%\bibitem{BDVV3}
%S. Biagi, S. Dipierro, E. Valdinoci, E. Vecchi,
%{\em A Faber-Krahn inequality for mixed local and nonlocal operators}, J. Anal. Math. {\bf 150}(2), (2023), 405--448.
 
%\bibitem{BDVV4}
%S. Biagi, S. Dipierro, E. Valdinoci, E. Vecchi,
%{\em A Hong-Krahn-Szeg\"{o} inequality for mixed local and nonlocal operators}, Math. Eng. {\bf 5} (1) (2023), Paper No. 014, 25 pp.
 
\bibitem{BDVV5}
S. Biagi, S. Dipierro, E. Valdinoci, E. Vecchi,
{\em A Brezis-Nirenberg type result for mixed local and nonlocal operators}, 
 NoDEA Nonlinear Differential Equations Appl. {\bf 32}, (2025), art. 62.

 \bibitem{BDVV6}
S. Biagi, S. Dipierro, E. Valdinoci and E. Vecchi,
{\em On a Sobolev critical problem for the superposition of a local and nonlocal operator with the ``wrong sign''},
Bull. Lond. Math. Soc. {\bf 58}(9), (2026), e70477. 

\bibitem{BEMV}
S. Biagi, F. Esposito, L. Montoro, E. Vecchi,
{\em On mixed local-nonlocal problems with Hardy potential},
published online (2025), Proc. Roy. Soc. Edinburgh Sect. A. 
\url{https://doi.org/10.1017/prm.2025.28}

%\bibitem{BMV}
%S. Biagi, D. Mugnai, E. Vecchi,
%{\em A Brezis-Oswald approach for mixed local and nonlocal operators},
%Commun. Contemp. Math. 2250057 (2022), 28 pp.

%\bibitem{BMV2}
%S. Biagi, D. Mugnai and E. Vecchi,
%{\em Necessary condition in a Brezis-Oswald-type problem for mixed local and nonlocal operators}, Appl. Math. Lett. {\bf 132}, (2022), 108177.
% 
\bibitem{BiPuVe}
S. Biagi, F. Punzo, E. Vecchi,
{\em Blow-up of solutions to semilinear parabolic equations driven by mixed local-nonlocal operators with large initial data}, preprint.
\url{https://arxiv.org/abs/2605.07911}

\bibitem{BiagiVecchi} 
S. Biagi, E. Vecchi,
{\em Multiplicity of positive solutions for mixed local-nonlocal singular critical problems}, 
Calc. Var. Partial Differ. Equ. {\bf 63}, (2024), Article 221.

\bibitem{BiagiVecchi2} 
S. Biagi, E. Vecchi,
{\em On the existence of a second positive solution to mixed local-nonlocal concave-convex critical problems}, 
Nonlinear Anal. {\bf 256}, (2025), 113795.
%\url{https://doi.org/10.1016/j.na.2025.113795}
%

\bibitem{ABiswas}
A. Biswas,
{\em The Pohozaev identity for mixed local-nonlocal operators},
J. Math. Anal. Appl. {\bf 557}(1), (2026), 130270.

%preprint.
%\url{https://arxiv.org/abs/2510.15507}
%
%
\bibitem{BFV}
M. Bonforte, A. Figalli, J.L. V\'{a}zquez,
{\em Sharp boundary behaviour of solutions to semilinear nonlocal elliptic equations},
Calc. Var.  Partial Differ. Equ. {\bf 57}, (2018), Article 57.

%\bibitem{BCP}
%J.-M. Bony, P. Courrège, P. Priouret, 
%{\em Semi-groupes de Feller sur une variété à bord compacte et problèmes aux limites intégro-différentiels du second ordre donnant lieu au principe du maximum},
%Ann. Inst. Fourier (Grenoble) {\bf 18}(2), (1968), 369--521.


%\bibitem{Brezis}
%H. Brezis,
%\emph{Functional analysis, Sobolev spaces and partial differential equations}, 
% U\-ni\-ver\-si\-text, Springer, New York, 2011. 

%\bibitem{BrezisLieb}
%H. Brezis, E. Lieb, 
%{\em A relation between pointwise convergence of functions and convergence of integrals}, 
%Proc. Amer. Math. Soc. {\bf 88}, (1983), 486--490.

%\bibitem{BrezisKato}
%H. Brezis, T. Kato, 
%{\em Remarks on the Schr\"{o}dinger operator with singular complex potentials}, 
%J. Math. Pures Appl. {\bf 58}, (1979), 137--151.

\bibitem{BrNir}
H. Brezis, L. Nirenberg,
{\em Positive solutions of nonlinear elliptic equations involving critical Sobolev exponents}, 
Comm. Pure Appl. Math. \textbf{36} (1983), no. 4, 437--477.

%\bibitem{BySo}
%S.-S. Byun, K. Song, 
%{\em Mixed local and nonlocal equations with measure data},
%Calc. Var. Partial Differential Equations  {\bf 62}, (2023), Article 14.

%\bibitem{CDV22}
%X. Cabr\'e, S. Dipierro, E. Valdinoci, 
%{\em The Bernstein Technique for Integro-Differential
%Equations}, Arch. Rational Mech. Anal. {\bf 243}, (2022), 1597--1652.

\bibitem{CaGiSp}
L.A. Caffarelli, B. Gidas, J. Spruck, 
{\em Asymptotic symmetry and local behavior of
semilinear elliptic equations with critical Sobolev growth}, 
Comm. Pure Appl. Math. {\bf 42}(3), (1989), 271--297.

\bibitem{Cancelier}
C. Cancelier, 
{\em Problèmes aux limites pseudo-différentiels donnant lieu au principe du maximum},
Comm. Partial Differential Equations {\bf 11}(15), (1986), 1677--1726.
%

\bibitem{CCMV}
N. Cangiotti, M. Caponi, A. Maione, E. Vitillaro, 
{\em Schr\"{o}dinger-Maxwell equations driven by mixed local-nonlocal operators}, 
Fract. Calc. Appl. Anal. {\bf 27}, (2024), 677--705.


\bibitem{CKSV}
 Z.-Q. Chen, P. Kim, R. Song, Z. Vondra\v{c}ek, 
 {\em Boundary Harnack principle for $\Delta + \Delta^{\alpha/2}$}, 
 Trans. Amer. Math. Soc. {\bf 364}(8), (2012), 4169--4205. 

\bibitem{COURR}
	P. Courr\`ege, \emph{G\'en\'erateur infinit\'esimal d'un semi-groupe de convolution sur $\R^n$, et formule de L\'evy-Khinchine},
	Bull. Sci. Math. (2) {\bf 88}, (1964), 3--30.

\bibitem{DPPQ}
L.M. Del Pezzo, B. Pichucho, A. Quaas,
{\em The Serrin exponent for Lane-Emden inequalities driven by sums of fractional $p$-Laplacians}, preprint. \url{https://arxiv.org/abs/2610.30415}

 \bibitem{DeFMin}
 C. De Filippis, G. Mingione, 
 {\em Gradient regularity in mixed local and nonlocal problems}, 
 Math. Ann. {\bf 388}, (2024), 261--328.


%

%\bibitem{daSilva}
%J.V. da Silva, A. Fiscella, V.A. Blanco Viloria,
%{\em Mixed local-nonlocal quasilinear problems with critical nonlinearities},
%J. Differential Equations  {\bf 408}, (2024), 494--536.

\bibitem{DPV}
E. Di Nezza, G. Palatucci, E. Valdinoci,
{\em Hitchhiker's guide to the fractional Sobolev spaces}, 
Bull. Sci. Math. {\bf 136}, (2012), 521--573.


%\bibitem{DiFaZh}
%M. Ding, Y. Fang, C. Zhang, C. 
%{\em Local behavior of the mixed local and nonlocal problems with nonstandard growth}, 
%J. London Math. Soc. {\bf 109}(6), (2024), Article e12947.

%\bibitem{DPLV1}
%S. Dipierro, E. Proietti Lippi, E. Valdinoci, 
%{\em Linear theory for a mixed operator with Neumann conditions}, Asymptotic Analysis {\bf 128}(4), (2022), 571--594.

\bibitem{DPLV2}
S. Dipierro, E. Proietti Lippi, E. Valdinoci, 
{\em (Non)local logistic equations with Neumann conditions}, Ann. Inst. H. Poincaré Anal. Non Linéaire (2022), 
DOI: 10.4171/AIHPC/57.

\bibitem{RE2}
S. Dipierro, K. Perera, C. Sportelli, E. Valdinoci, \emph{An existence theory for superposition operators of mixed order subject to jumping nonlinearities}, Nonlinearity {\bf{37}}, (2024), no. 5, Paper
No. 055018, 27 pp. 

\bibitem{RE22}
S. Dipierro, K. Perera, C. Sportelli, E. Valdinoci, \emph{An existence theory for nonlinear superposition operators of mixed fractional order}, Commun. Contemp. Math. {\bf{27}}, (2025), no. 8, Paper No. 55000, 39 pp.


\bibitem{RE5}
S. Dipierro, E.  Proietti Lippi, C. Sportelli, E. Valdinoci, \emph{A general theory for the
$(s, p)$-superposition of nonlinear fractional operators}, Nonlinear Anal. Real World Appl. {\bf 82}, (2025), Paper
No. 104251, 24 pp.


\bibitem{RE6}
S. Dipierro, E. Proietti Lippi, C. Sportelli, E. Valdinoci, \emph{
Logistic diffusion equations governed by the superposition of operators of mixed fractional order}, 
Ann. Mat. Pura Appl. (4) {\bf 205}, (2026), no. 2, 539--589.


\bibitem{RE7}
S. Dipierro, E. Proietti Lippi, C. Sportelli, E. Valdinoci, \emph{Maximum principles and
spectral analysis for the superposition of operators of fractional order}, 
Matematica {\bf 5} (2026), no. 2, Paper No. 35, 31 pp.


\bibitem{emBE}
S. Dipierro, E. Proietti Lippi, C. Sportelli, E. Valdinoci, 
{\em Optimal embedding results for fractional Sobolev spaces},
Math. Z. {\bf 313}, (2026), no. 3, Paper No. 49.

\bibitem{RE8}
S. Dipierro, E. Proietti Lippi, C. Sportelli, E. Valdinoci, \emph{Nonlocal eigenvalue
problems and superposition operators}, 
Discrete Contin. Dyn. Syst. {\bf 58}, (2027), 417--448.

\bibitem{DV-niche}
S. Dipierro, E. Valdinoci,
{\em Description of an ecological niche for a mixed local/nonlocal dispersal: an evolution equation and a new Neumann condition arising from the superposition of Brownian and L\'{e}vy processes},
Phys. A. {\bf 575}, (2021), 126052.


%\bibitem{LLAVE09} 
%R. de la Llave, E. Valdinoci, \emph{A generalization of Aubry-Mather theory to partial differential equations and pseudo-differential equations}, Ann. Inst. H. Poincar\'e C Anal. Non Lin\'eaire {\bf 26}, (2009), no. 4, 1309--1344.
	
	
\bibitem{FerrariVerbitsky}
F.~Ferrari and I.~E.~Verbitsky,
\emph{Radial fractional Laplace operators and Hessian inequalities},
J. Differential Equations \textbf{253}(1), (2012), 244--272.	

\bibitem{GarainKinnunen}
P. Garain, J. Kinnunen,
{\em On the regularity theory for mixed local and nonlocal quasilinear elliptic equations}, 
Trans. Amer. Math. Soc. {\bf 375}(8), (2022), 5393--5423.

\bibitem{GarainLindgren}
P. Garain, E. Lindgren, 
{\em Higher Hölder regularity for mixed local and nonlocal degenerate elliptic equations}, 
Calc. Var. Partial Differ. Equ.{\bf 62}, (2023), Article 67.

%\bibitem{GARR} 
%M. G. Garroni, J. L. Menaldi, {\em
%Second order elliptic integro-differential problems},
%Chapman \& Hall/CRC Research Notes in Mathematics,  Boca Raton, FL, 2002. xvi+221 pp.

\bibitem{GidasSpruck}
B. Gidas, J. Spruck, 
{\em Global and local behavior of positive solutions of nonlinear
elliptic equations}, 
Comm. Pure Appl. Math. {\bf 34}(4), (1981), 525--598.

\bibitem{GMV}
G. Giovannardi, D. Mugnai, E. Vecchi,
{\em An Ahmad-Lazer-Paul-type result for indefinite mixed local-nonlocal problems}, 
J. Math. Anal. Appl. {\bf 527}(2), (2023), 127442.


\bibitem{GuoLiXie}
Y. Guo, C. Li, J. Xie,
{\em Liouville Theorems for the Lane-Emden Equation Involving a Mixed Local-Nonlocal Operator}, preprint.
\url{https://arxiv.org/abs/2606.16940}

%\bibitem{JA06}
%	E. R. Jakobsen, K. H. Karlsen, \emph{A ``maximum principle for semicontinuous functions'' applicable to integro-partial differential equations},
%	NoDEA Nonlinear Differential Equations Appl. {\bf 13}, (2006), no. 2, 137--165.

\bibitem{MMV} 
A. Maione, D. Mugnai, E. Vecchi,
{\em Variational methods for nonpositive mixed local--nonlocal operators}, 
Fract. Calc. Appl. Anal. {\bf 26}, (2023), 943--961.

\bibitem{MMR}
D. Mazzoleni, R. Moraschi, B. Ruffini,
{\em Existence and Regularity in the Small-Mass Regime for a Hartree--Ohta-Kawasaki Shape Optimization Problem}, preprint.
\url{https://arxiv.org/abs/2604.13895}

\bibitem{MiPo}
E. Mitidieri, S.I. Pohozaev,
{\em Towards a Unified Approach to Nonexistence of Solutions for a Class of Differential Inequalities}, 
Milan J. Math. {\bf 72}, (2004), 129--162.

%\bibitem{LeoniSob}
%G. Leoni,
%{\em A First Course in Sobolev Spaces},
%Grad. Stud. Math. 105,
%American Mathematical Society, Providence, RI, 2009, xvi+607 pp.

%%\bibitem{LeoniFract}
%%G. Leoni,
%%\emph{A First Course in Fractional Sobolev Spaces}, 
%%Grad. Stud. Math. 229,
%%American Mathematical Society, Providence, RI, 2023, xv+586 pp.

\bibitem{PESP:2024}
K. Perera, C. Sportelli, \emph{A multiplicity result for critical elliptic problems involving differences of local and nonlocal operators}, Topol. Methods Nonlinear Anal. {\bf 63}, (2024), no. 2, 717--731.

\bibitem{RSPoho}
X. Ros-Oton,  J. Serra,
{\em Nonexistence Results for Nonlocal Equations with Critical and Supercritical Nonlinearities},
Comm. Partial Differential Equations \textbf{40} (2015), no. 1, 115--133.

%\bibitem{RE10}
%E. Proietti Lippi, C. Sportelli, \emph{Ground state solution for the Choquard equation under the superposition of operators of mixed fractional order}, Fract. Calc. Appl. Anal. {\bf 29}, (2026), no. 2, 708--742.

%\bibitem{SeVa}
%R. Servadei, E. Valdinoci,
%{\em The Brezis-Nirenberg result for the fractional Laplacian}, 
%Trans. Amer. Math. Soc. {\bf 367}(1), (2015), 67--102.

%\bibitem{SILVE06} 
%L. Silvestre, \emph{H\"older estimates for solutions of integro-differential equations like the fractional Laplace},
%	Indiana Univ. Math. J. {\bf 55}, (2006), no. 3, 1155--1174.

%\bibitem{Stampacchia}
%G. Stampacchia,
%{\em Le probl\`{e}me de Dirichlet pour les \'{e}quations elliptiques du second ordre \`{a} coefficients discontinus}, 
%Ann. Inst. Fourier (Grenoble) {\bf 15}, (1965), 189--258.

\bibitem{SVWZ} 
X. Su, E. Valdinoci, Y. Wei, J. Zhang,
{\em Regularity results for solutions of mixed local and nonlocal elliptic equations}, Math. Z. {\bf 302}, (2022), 1855--1878.

\bibitem{SVWZ2} 
X. Su, E. Valdinoci, Y. Wei, J. Zhang,
{\em On Some Regularity Properties of Mixed Local and Nonlocal Elliptic Equations}, 
J. Differential Equations {\bf 416}, (2025), 576--613.

%\bibitem{Talenti}
%G. Talenti, 
%{\em Elliptic equations and rearrangements}, 
%Ann. Scuola Norm. Sup. Pisa Cl. Sci. {\bf 3}, (1976), 697--718.

%\bibitem{Talenti2}
%G. Talenti, 
%{\em Best constant in Sobolev inequality}, 
%Ann. Mat. Pura Appl. {\bf 110}, (1976), 353--372.

%\bibitem{Talenti3}
%G. Talenti, Inequalities in rearrangement invariant function spaces, in {\it Nonlinear analysis, function spaces and applications, Vol.\ 5 (Prague, 1994)}, 177--230, Prometheus, Prague.

 
 \end{thebibliography}
\end{document}